\documentclass[10pt,a4paper]{amsart}
\pdfoutput=1

\usepackage{lmodern}                    
\usepackage[T1]{fontenc}
\usepackage[utf8]{inputenc}

\usepackage[british]{babel}

\usepackage{amsmath,amssymb,amsthm,amsfonts,amsbsy,latexsym}
\usepackage[abbrev,alphabetic]{amsrefs}
\usepackage{mathtools}
\usepackage{commath}

\usepackage{url}
\usepackage[breaklinks,linktocpage]{hyperref}
\usepackage[capitalise]{cleveref}

\newcommand*{\doi}[1]{\href{https://doi.org/#1}{doi:#1}}
\usepackage{xifthen}
\newcommand*{\msc}[2][]{\href{https://mathscinet.ams.org/mathscinet/search/mscdoc.html?code=#2,(#1)}{Primary: #2\ifthenelse{\isempty{#1}}{}{; Secondary: #1}.}}

\usepackage{tikz}
\usetikzlibrary{arrows}
\usetikzlibrary{babel}
\usetikzlibrary{cd}
\usetikzlibrary{decorations.pathreplacing}
\usetikzlibrary{matrix}
\usetikzlibrary{quotes}

\makeatletter
\protected\def\tikz@nonactivecolon{\ifmmode\mathrel{\mathop\ordinarycolon}\else:\fi}
\makeatother

\usepackage{todonotes} 
\usepackage{enumitem}
\usepackage{csquotes}
\usepackage{verbatim}

\usepackage{etoolbox}
\apptocmd{\sloppy}{\hbadness 10000\relax}{}{}
\usepackage{microtype}
\usepackage{scrhack}
\newcommand*{\smashedvdots}{\ensuremath{\vphantom{\int^0}\smash\vdots}}

\theoremstyle{definition}
\newtheorem{defi}{Definition}[section]

\theoremstyle{plain}
\newtheorem{lemm}[defi]{Lemma}
\crefname{lemm}{Lemma}{Lemmata}
\newtheorem{theo}[defi]{Theorem}
\crefname{theo}{Theorem}{Theorems}
\newtheorem*{theo*}{Theorem}
\crefname{theo*}{Theorem}{Theorems}
\newtheorem{coro}[defi]{Corollary}
\newtheorem{prop}[defi]{Proposition}
\newtheorem{conj}[defi]{Conjecture}

\crefname{theoenum}{Theorem}{Theorems}

\theoremstyle{remark}
\newtheorem{rema}[defi]{Remark}
\newtheorem{exam}[defi]{Example}

\newtheoremstyle{maintheorem}{}{}{\itshape}{}{\bfseries}{}{.5em}{#1 \!\thmnote{#3}.}
\theoremstyle{maintheorem}
\newtheorem*{mainthm}{Theorem}

\newcommand{\overbar}[1]{\mkern 1.5mu\overline{\mkern-1.5mu#1\mkern-1.5mu}\mkern
1.5mu}
\newcommand*{\NN}{\mathbb{N}}
\newcommand*{\ZZ}{\mathbb{Z}}
\newcommand*{\QQ}{\mathbb{Q}}
\newcommand*{\RR}{\mathbb{R}}

\let\leq\leqslant
\let\geq\geqslant
\let\phi\varphi

\DeclareMathOperator{\im}{im}

\DeclareMathOperator{\id}{id}
\DeclareMathOperator{\coker}{coker}

\DeclareMathOperator{\supp}{supp}
\DeclareMathOperator{\pr}{pr}
\DeclareMathOperator{\rk}{rk}
\DeclareMathOperator{\Aut}{Aut}
\DeclareMathOperator{\res}{res}

\DeclareMathOperator{\GL}{GL}

\DeclareMathOperator{\thickness}{th}
\DeclareMathOperator{\Ore}{Ore}
\DeclareMathOperator{\tors}{tors}

\newcommand*{\Kr}{\ensuremath{{\widetilde K}_1}}

\newcommand*{\Poly}{{\ensuremath{\mathcal{P}}}}
\newcommand*{\PolyT}{{\ensuremath{\mathcal{P}_T}}}

\newcommand*{\Polytope}{\mathcal{P}}
\newcommand*{\MarkedPolytope}{\mathcal{M}}
\newcommand*{\Dab}[1][D]{{#1}^\times_{\textrm{ab}}}
\newcommand*{\bag}[1][D]{b^{#1}}
\newcommand*{\hag}[1][D]{h^{#1}}
\newcommand*{\chiag}[1][D]{\chi^{#1}}
\newcommand*{\Pag}[1]{P^{#1}}
\newcommand*{\PLtwo}{P_{L^2}}
\newcommand*{\torsionag}[1][D]{\rho_{#1}}
\newcommand*{\linnelld}{\mathcal{D}}

\begin{document}

\title{The agrarian polytope of two-generator one-relator groups}

\author{Fabian Henneke}
\email{\href{mailto:henneke@uni-bonn.de}{henneke@uni-bonn.de}}
\address{Max-Planck-Institut für Mathematik, Vivatsgasse 7, D-53111 Bonn, Germany}

\author{Dawid Kielak}
\email{\href{mailto:dkielak@math.uni-bielefeld.de}{dkielak@math.uni-bielefeld.de}}
\address{Fakultät für Mathematik, Universität Bielefeld, Postfach 100131, D-33501 Bielefeld, Germany}


\begin{abstract}

Relying on the theory of agrarian invariants introduced in previous work, we solve a conjecture of Friedl--Tillmann: we show that the marked polytopes they constructed for two-generator one-relator groups with nice presentations are independent of the presentations used. We also show that, when the groups are additionally torsion-free, the agrarian polytope encodes the splitting complexity of the group. This generalises theorems of Friedl--Tillmann and Friedl--Lück--Tillmann.
\end{abstract}

\maketitle


\section{Introduction}

A focal point of much activity in low-dimensional topology in the recent years was the Virtually Fibred Conjecture of Thurston. The conjecture, now confirmed by Agol~\cite{Agol2013}, stipulated that every (closed connected oriented) hyperbolic $3$-manifold virtually fibres over the circle. Thanks to a classical result of Stallings~\cite{Stallings1962}, the statement can be recast in the language of group theory:
\begin{theo*}[\cite{Agol2013}]
 Let $G$ be the fundamental group of a closed connected oriented hyperbolic $3$-manifold. Then $G$ admits a finite index subgroup which maps onto $\ZZ$ with a finitely generated kernel.
\end{theo*}

The study of finiteness properties of kernels of epimorphisms to $\ZZ$ is the cornerstone of the Bieri--Neumann--Strebel theory. In particular, the specific question of which epimorphisms $\phi \colon G \to \ZZ$ have finitely generated kernels is encoded by the first BNS invariant $\Sigma^1(G)$, a subset of $H^1(G;\RR)$.

If $G$ is the fundamental group of a connected orientable $3$-manifold, then $\Sigma^1(G)$ is controlled by the Thurston polytope (see~\cites{Thurston1986,Bierietal1987}). More explicitly, there exists a compact convex polytope $P \subset H_1(G;\RR)$ with some vertices marked, such that an epimorphism $\phi \colon G \to \ZZ$ belongs to $\Sigma^1(G)$ if and only if it attains its minimum when restricted to $P$ uniquely at a marked vertex. In this case, the kernel has to be a surface group, and the thickness of the polytope $P$ in the direction of $\phi$, denoted $\thickness_\phi(P)$,  gives us the genus of the surface.

\smallskip
A similar picture was conjectured by Friedl--Tillmann~\cite{FT2015} to hold for two-generator one-relator groups. They start with a \emph{nice} presentation $\pi$ of such a group $G$, which in particular requires $H_1(G)$ to be of rank $2$, and using the presentation they construct a polytope $\Polytope_\pi \subset H_1(G;\RR)$. Then they mark some of the vertices of $\Polytope_\pi$, and obtain a marked polytope $\MarkedPolytope_\pi$ which controls $\Sigma^1(G)$ in a way analogous to the Thurston polytope.
The process of obtaining $\MarkedPolytope_\pi$ is very similar to Brown's algorithm~\cite{Brown1987}, a method of computing $\Sigma^1(G)$ of one-relator groups.

Friedl--Tillmann made two conjectures related to $\MarkedPolytope_\pi$: First, they conjectured that the polytope $\MarkedPolytope_\pi$ depends only on $G$ and not on $\pi$; second, the thickness $\thickness_\phi(\MarkedPolytope_\pi)$ for an epimorphism $\phi \colon G \to \ZZ$ is supposed to compute the \emph{(free) splitting complexity} of $G$ relative to $\phi$, a number informing about the `smallest' way $G$ can be written as an HNN extension with induced character $\phi$.
They proved their conjectures in \cite{FT2015} under the additional hypothesis that the group $G$ is residually \{torsion-free elementary amenable\}; later the first conjecture was confirmed by Friedl--L\"uck~\cite{FL2017} under the weaker assumption that $G$ is torsion-free and satisfies the strong Atiyah conjecture.

Here a complete resolution of the first conjecture is offered:
\begin{mainthm}[\ref{theo:tgor:invariance}]
If $G$ is a group admitting a nice $(2, 1)$-presentation $\pi$, then $\MarkedPolytope_\pi\subset H_1(G;\RR)\cong \RR^2$ is an invariant of $G$ (up to translation). Moreover, if $G$ is torsion-free then $\Polytope_\pi=\Pag{D_r}(G)$ for any choice of an agrarian embedding $\ZZ G\hookrightarrow D$.
\end{mainthm}
The notation $\Pag{D_r}(G)$ stands for the \emph{agrarian} polytope, as introduced in \cite{HK2019}, defined over the rationalisation $D_r$ of a skew field $D$. In fact, $\Pag{D_r}(G)$ is an invariant defined for any torsion-free two-generator one-relator group $G$ other than the free group on two generators, even if $b_1(G)=1$.

The second conjecture is also confirmed, assuming that $G$ is torsion-free:
\begin{mainthm}[\ref{theo:tgor:thickness}]
Let $G$ be a torsion-free two-generator one-relator group other than the free group on two generators. Then for every epimorphism $\phi\colon G\to\ZZ$ we have
\[c(G,\phi)=c_f(G, \phi)=\thickness_\phi(\Pag{D_r}(G)) + 1.\]
\end{mainthm}
Here, $c(G,\phi)$ stands for the splitting complexity, and $c_f(G, \phi)$ for the free splitting complexity.

Both of these theorems are proven using the machinery of agrarian invariants, introduced by the authors in~\cite{HK2019}.

(After the first version of this article appeared, Jaikin-Zapirain and L\'opez-\'Alvarez~\cite{JL2019} published a proof of the strong Atiyah conjecture for torsion-free one-relator groups.
This provides an alternative proof of the torsion-free case of our results as remarked in~\cite[Remark~5.5]{FL2017} and~\cite[Theorem~5.2]{FLT2016}).

\subsection*{Acknowledgements}
The authors are grateful to Alan Logan for pointing out that~\cref{theo:tgor:invariance} in the case of a group with torsion follows from a result of Pierce.
The first author would further like to thank his advisor Wolfgang Lück as well as Stefan Friedl and Xiaolei Wu for helpful discussions. He is especially grateful to Stefan Friedl for an invitation to Regensburg and the hospitality experienced there.

The present work is part of the first author's PhD project at the University of Bonn. He was supported by Wolfgang Lück's ERC Advanced Grant ``KL2MG-interactions'' (no. 662400) granted by the European Research Council. The second author was supported by the grant KI 1853/3-1 within the Priority Programme 2026 \href{https://www.spp2026.de/}{`Geometry at Infinity'} of the German Science Foundation (DFG).

\section{Agrarian invariants}
The second author introduced the notion of an agrarian group in~\cite{Kielak2018}.
In~\cite{HK2019}, the authors then developed a theory of algebraic invariants of nice spaces with an action of an agrarian group, which proceeds in analogy to the construction of $L^2$-invariants.
In this section, we will review the constructions and properties of these invariants, namely agrarian Betti numbers, agrarian torsion and agrarian polytopes, inasmuch as they are relevant to the proofs of our main results.
For a full introduction, which also contains comparisons to $L^2$-invariants and a discussion of the dependence of agrarian invariants on the choice of an agrarian embedding, we refer the reader to~\cite{HK2019}.
We will mostly follow the presentation therein, but use a different approach to the definition of agrarian torsion that is better suited for our computational purposes.

\subsection{Agrarian groups and associated Ore embeddings}

The key player in our story will be an integral group ring $\ZZ G$.
Throughout the paper, all tensor products will be understood to be taken over $\ZZ G$ unless explicitly indicated otherwise.

\begin{defi}
    Let $G$ be a non-trivial group.
    An \emph{agrarian embedding for} $G$ is an injective ring homomorphism $\alpha\colon \ZZ G\hookrightarrow D$ with $D$ a skew field.
    If $G$ admits an agrarian embedding (into a skew field $D$), it is called a \emph{($D$-)agrarian group}.
\end{defi}

An agrarian group is always torsion-free.
Examples of agrarian groups are given by torsion-free groups satisfying the Atiyah conjecture over $\QQ$~\cite[Theorem~10.39]{Luck2002} as well as by torsion-free one-relator groups~\cite{LL1978}.
For a more detailed discussion of examples and the inheritance properties enjoyed by agrarian groups, see \cite{Kielak2018}.

In order to construct new agrarian embeddings out of given ones, we will need to consider twisted group rings:
\begin{defi}
Let $R$ be a ring and let $G$ be a group. Let functions $c\colon G\to\Aut(R)$ and $\tau \colon G\times G\to R^\times$ be such that
\begin{align*}
    c(g)\circ c(g') &= c_{\tau(g, g')} \circ c(gg')\\
    \tau(g,g')\tau(gg', g'') &= c(g)(\tau(g', g''))\tau(g, g'g''),
\end{align*}
where $g,g', g''\in G$, and where $c_r\in \Aut(R)$ for $r\in R^\times$ denotes the conjugation map $x\mapsto rxr^{-1}$. The functions $c$ and $\tau$ are called \emph{structure functions}. We denote by $RG$ the free $R$-module with basis $G$ and write elements of $RG$ as finite $R$-linear combinations $\sum_{g\in G} \lambda_g\ast g$ of elements of $G$. When convenient, we shorten $1\ast g$ to $g$. The structure functions endow $RG$ with the structure of an (associative) \emph{twisted group ring} by declaring
\[
g \cdot (r \ast 1) = c(g)(r) \ast g \textrm{ and }
g \cdot g'= \tau(g,g') \ast gg'
\]
and extending linearly.
\end{defi}

The usual, \emph{untwisted} group ring is obtained from the definition by taking the structure functions to be trivial.
In the following, group rings with $R=\ZZ$ will always be understood to be untwisted.

The fundamental example of a twisted group ring arises in the following way:
\begin{exam}
\label{exam:agrarian:twisted}
Let $\phi\colon G\twoheadrightarrow H$ be a group epimorphism with kernel the normal subgroup $K\leq G$.
We choose any section $s\colon H\to G$ of the map of sets underlying $\phi$, i.e., a map such that $\phi \circ s = \id_H$.
We denote by $(\ZZ K)H$ the twisted group ring defined by the structure functions $c(h)(r)=s(h)rs(h)^{-1}$ and $\tau(h, h')=s(h)s(h')s(hh')^{-1}$.
The untwisted group ring $\ZZ G$ is then isomorphic to the twisted group ring $(\ZZ K)H$ via the map
\[g\mapsto \left(g \cdot (s\circ\phi)(g)^{-1}\right)\cdot\phi(g).\]
\end{exam}

The twisted group ring construction will enable us  to construct out of a given agrarian embedding for a group $G$ new agrarian embeddings with better properties.

Recall that a ring $R$ without non-trivial zero divisors satisfies the \emph{Ore condition} if for every $p,q \in R$ with $q \neq 0$ there exists $r,s \in R$ with $s \neq 0$ such that
\[
ps = qr.
\]
This identity enables the conversion of a left fraction $ q^{-1} p$ into a right fraction $r s^{-1}$, which in turn makes it possible to multiply fractions (in the obvious way). The Ore condition also guarantees the existence of common denominators, and thus allows for addition of fractions. Thanks to these properties, the ring $R$ embeds into its \emph{Ore field of fractions}
\[\Ore(R)\coloneqq\{ q^{-1} p \mid p,q \in R, q \neq 0 \},\]
which is evidently a skew field.
We refer the reader to the book of Passman~\cite[Section 4.4]{Passman1985} for details and proofs.

\begin{lemm}
\label{lemm:agrarian:ore_field}
Let $\alpha\colon \ZZ G\to D$ be an agrarian embedding for a finitely generated group $G$, and let $K\leq G$ be a normal subgroup such that $H\coloneqq G/K$ is free abelian.
Then $\alpha$ induces an injective ring homomorphism
\[(\ZZ K)H \hookrightarrow DH,\]
where $(\ZZ K)H$ is as defined in \cref{exam:agrarian:twisted}, and $DH$ is a twisted group ring with the same structure functions as $(\ZZ K)H$. Furthermore, $DH$ admits an Ore field of fractions $\Ore(DH)$ and we obtain an agrarian embedding
\[\alpha_K\colon \ZZ G\cong (\ZZ K)H\hookrightarrow DH\hookrightarrow \Ore(DH),\]
which we call the \emph{$K$-rationalisation} of $\alpha$.
\end{lemm}
\begin{proof}
See \cite[Definition~2.6]{HK2019} and the preceding discussion.
\end{proof}

Observe that while the map $\alpha_K$ certainly depends on the choice of a section of the projection $G\to G/K$, it follows from \cite[Lemma~2.5]{HK2019} that the target skew field is unique up to isomorphism.
For the purposes of this paper, we will assume that such a section has been chosen once and for all for any group under consideration, and therefore always speak of \emph{the} $K$-rationalisation of an agrarian embedding for $G$.

The smallest choice for $K$ in \cref{lemm:agrarian:ore_field} is clearly the kernel of the projection of $G$ onto the free part of its abelianisation.
Since the $K$-rationalisation for this particular choice of $K$ will be most useful for us, we introduce special notation for it:
\begin{defi}
\label{defi:agrarian:rationalisation}
Let $\alpha\colon \ZZ G\to D$ be an agrarian embedding for a finitely generated group $G$.
Further let $H$ be the free part of the abelianisation of $G$ and $K$ the kernel of the projection of $G$ onto $H$.
The $K$-rationalisation of $\alpha$ for this particular choice of $K$ is simply called the \emph{rationalisation} and is denoted by $\alpha_r$.
The target skew field of $\alpha_r$ is also denoted by $D_r$.
\end{defi}

The following lemma essentially states that taking iterated ``partial'' rationalisations with respect to a chain $K\leq K'\leq G$ of normal subgroups is naturally equivalent to the ``full'' rationalisation:
\begin{lemm}
\label{lemm:agrarian:ore_comparison}
Let $G$ be a finitely generated agrarian group with agrarian embedding $\alpha\colon \ZZ G\hookrightarrow D$. Denote by $\pr\colon G\to H$ the projection onto the free part $H$ of the abelianisation of $G$. Let $\phi\colon G\to H'$ be an epimorphism onto a finitely generated free abelian group, inducing the following commutative diagram of epimorphisms:
\[
\begin{tikzcd}
G\arrow[r,"\pr",two heads]\arrow[dr,"\phi",swap,two heads] & H\arrow[d,"\overbar\phi",two heads]\\
&H'
\end{tikzcd}
\]
Denote the kernels of $\pr$, $\phi$ and $\overbar\phi$ by $K$, $K_\phi$ and $K_{\overbar\phi}$, respectively. Further let $s$ and $t$ be sections of the epimorphisms $\pr$ and $\overbar \phi$, respectively. Then
\begin{align*}
\beta\colon(DK_{\overbar\phi})H'&\to DH\\
\sum_{h'\in H'}\Big (\sum_{k\in K_{\overbar\phi}} u_{k, h'}\ast k\Big )\ast h'&\mapsto \sum_{\substack{\mathllap{h'}\in \mathrlap{H'} \\ \mathllap{k}\in \mathrlap{K_{\overbar\phi}}}} u_{k,h'} \ast k t(h')
\end{align*}
is an isomorphism between twisted group rings constructed using the sections $s$, $t$ and $s\circ t$. It extends to an isomorphism
\[
\beta\colon\Ore(\Ore(DK_{\overbar\phi})H')\xrightarrow{\cong}\Ore(DH)
\]
of skew fields.
\end{lemm}
\begin{proof}
Left $D$-bases of $(DK_{\overbar\phi})H'$ and $DH$ are given by $k\ast h'$ and $kt(h')$ respectively for $k\in K_{\overbar\phi}$ and $h'\in H'$. These bases are identified bijectively by $\beta$ with inverse $h\mapsto h t(\overbar\phi(h)^{-1})\ast\overbar\phi(h)$. It follows that $\beta$ is an isomorphism of left $D$-modules. Checking that $\beta$ respects the twisted group ring multiplication is a tedious but direct computation that we will omit.

Since $DK_{\overbar\phi}$ is a subring of $DH$, and since the rings have no non-trivial zero divisors, $\beta$ extends to an injection $\Ore(DK_{\overbar\phi})H'\hookrightarrow \Ore(DH)$ that contains $DH$ in its image. Ore localising again, this implies that $\beta$ extends to an isomorphism $\Ore(\Ore(DK_{\overbar\phi})H')\to \Ore(DH)$.
\end{proof}

\subsection{Agrarian Betti numbers}

Given an agrarian embedding $\ZZ G\hookrightarrow D$ for a group $G$, we can associate to any $\ZZ G$-chain complex the $D$-dimensions of its $D$-homology groups, which can be viewed as equivariant analogues of Betti numbers:

\begin{defi}
    Let $G$ be an agrarian group with a fixed agrarian embedding $\alpha\colon G\hookrightarrow D$.
    For a $\ZZ G$-chain complex $C_*$ and $n\in\ZZ$, the \emph{$n$-th $D$-Betti number} of $C_*$ with respect to the agrarian embedding $\alpha$ is defined as
    \[\bag_n(C_*)\coloneqq \dim_D H_n(D\otimes C_*)\in \NN\sqcup \{\infty\},\]
    where $D$ becomes a right $\ZZ G$-module via $\alpha$.
    If $\bag_n(C_*)=0$ for all $n\in\ZZ$, then $C_*$ is called \emph{$D$-acyclic}.
\end{defi}

We will usually consider agrarian Betti numbers of suitably well-behaved spaces with an action of an agrarian group $G$. Recall that a \emph{$G$-CW-complex} is a CW-complex with a (left) $G$-action that maps $p$-cells to $p$-cells in such a way that any cell mapped into itself is already fixed pointwise. A $G$-CW-complex is called \emph{free} if its $G$-action is free. A $G$-orbit of a cell in the underlying CW-complex is called a \emph{$G$-cell}, with respect to which we understand the qualifiers \emph{finite} and $\emph{of finite type}$. Note that the cellular chain complex of a G-CW-complex naturally has the structure of a (left) $\ZZ G$-chain complex.

If we take $C_*$ to be the cellular $\ZZ G$-chain complex of a $G$-CW-complex, we obtain a notion of agrarian Betti numbers for such spaces.
It turns out that these invariants satisfy most of the well-known properties of non-equivariant Betti numbers.
For example, at least for finite free $G$-CW-complexes, they are homotopy invariant, compute the same Euler characteristic and are bounded from above by the number of equivariant cells.
They also behave similarly to $L^2$-Betti numbers as they vanish in dimension 0 and, if $G$ is amenable, in every dimension.
As these properties will not be used in the present work, we refer the reader to~\cite[Theorem~3.9]{HK2019} for the precise statements.

\subsection{Agrarian torsion}

Let $G$ be an agrarian group with a fixed agrarian embedding $\alpha \colon \ZZ G\hookrightarrow D$.
We write $D^\times$ for the group of units of $D$ and denote its abelianisation by $\Dab$.
The canonical projection $D^\times \to\Dab$ can be extended uniquely to a non-commutative notion of a determinant, the \emph{Dieudonn\'{e} determinant}, as follows.
We denote by $\GL(D)$ the group of all finite invertible matrices with entries in $D$, where every matrix is identified with any matrix obtained from it by adding an identity block in the bottom-right corner.
Then by~\cite[Theorem~2.2.5]{Rosenberg1994}, there is a unique group homomorphism $\det_D\colon \GL(D)\to\Dab$ with the following properties:
\begin{enumerate}
    \item $\det_D$ is invariant under elementary row operations;
    \item $\det_D$ maps the identity matrix to 1;
    \item $\det_D(\mu\cdot A)=\overbar\mu \cdot \det_D(A)$ for $A\in\GL(D)$ and $\mu\in D^\times$ with image $\overbar\mu\in \Dab$.
\end{enumerate}

If $C_*$ is now a finite free $\ZZ G$-chain complex that is $D$-acyclic with respect to $\alpha$, then the $D$-chain complex $D\otimes_{\ZZ G} C_*$ will be contractible.
In~\cite{HK2019}, the agrarian torsion $\torsionag(C_*)$ of such a chain complex $C_*$ together with a choice of a basis was defined  as a non-commutative $\Dab$-valued Reidemeister torsion in the sense of~\cite{Cohen1973}.
First, out of a chain contraction of $C_*$,an element of the reduced $K$-group $\Kr(D)$ is constructed, which is then mapped to $\Dab$ via a map induced by the Dieudonn\'{e} determinant of $D$.
For the details of this definition, we refer the reader to~\cite[Section~4]{HK2019}.

While the construction of agrarian torsion in~\cite{HK2019} is well-suited for the comparison to $L^2$-torsion, for our current purposes a slightly different way of computing agrarian torsion is more convenient.

We will use concepts and notation from~\cite[I.2.1]{Turaev2001}. Assume that we are given a $D$-acyclic finite free $\ZZ G$-chain complex $C_*$ concentrated in degrees $0$ through $m$, which is equipped with a choice of a preferred basis.
By fixing an ordering of the preferred basis, we identify subsets of $\{1, \dots, \rk C_p\}$ with subsets of the preferred basis elements of $C_p$.
We then denote by $A_p$, for $p=0,\dots, m-1$, the matrix representing the differential $c_{p+1}\colon C_{p+1}\to C_p$ in the preferred bases.
Note the shift in grading between $A_p$ and $c_{p+1}$, which is needed in order to bring our notation in line with that of Turaev.
The matrix $A_p$ consists of the entries $a^p_{jk}\in \ZZ G$, where $j=1,\dots,\rk C_{p+1}$ and $k=1,\dots,\rk C_p$.

\begin{defi}
A \emph{matrix chain} for $C_*$ is a collection of sets $\gamma=(\gamma_0,\dots, \gamma_m)$, where $\gamma_p\subseteq \{1, \dots, \rk C_p\}$ and $\gamma_0=\emptyset$. Write $S_p=S_p(\gamma)$ for the submatrix of $A_p$ formed by the entries $a^p_{jk}$ with $j\in\gamma_{p+1}$ and $k\not\in\gamma_p$.
A matrix chain $\gamma$ is called a \emph{$\tau$-chain} if $S_p$ is a square matrix for $p=0,\dots, m-1$.
A $\tau$-chain $\gamma$ is called \emph{non-degenerate} if $\det_D(S_p)\neq 0$ for all $p=0,\dots,m-1$.
\end{defi}

We want to point out that the reference~\cite[I.2.1]{Turaev2001} only considers chain complexes over a commutative field $\mathbb{F}$.
Nonetheless, all statements and proofs directly carry over to our setting of chain complexes over a skew field $D$ if we throughout replace the commutative determinant $\det_{\mathbb{F}}\colon \GL(\mathbb{F})\to \mathbb{F}^\times$ with the Dieudonn\'{e} determinant $\det_D$.
In particular, there is still a well-behaved notion of the \emph{rank} of a matrix $A$ over a skew field $D$, which can be defined in any of the following equivalent ways:
\begin{itemize}
    \item the largest number $r$ such that $A$ contains an invertible $r\times r$-submatrix;
    \item the $D$-dimension of the image of the linear map of left $D$-vector space given by right multiplication by $A$;
    \item the $D$-dimension of the right $D$-vector space spanned by the columns of $A$ (the \emph{column rank});
    \item the $D$-dimension of the left $D$-vector space spanned by the rows of $A$ (the \emph{row rank}).
\end{itemize}
With this convention, the proofs in~\cite[I.2.1]{Turaev2001} carry over verbatim.

Taken together, Theorem~I.2.2 and Remark~I.2.7 in \cite{Turaev2001} imply that any non-degenerate $\tau$-chain can be used to compute the agrarian torsion of $C_*$ as defined in \cite[Definition~4.7]{HK2019} and such a $\tau$-chain always exists if the complex is $D$-acyclic.
Note though that Turaev's convention for torsion differs from the one used in \cite{HK2019} in that he writes torsion multiplicatively instead of additively and uses the inverse of the torsion element in $\Kr(D)$ we construct, see~\cite[Theorem~I.2.6]{Turaev2001}.
Correcting for these differences by inserting a sign, we obtain
\begin{theo}
\label{theo:twisted:matrix_chain}
For any non-degenerate $\tau$-chain $\gamma$ of a $D$-acyclic finite free $\ZZ G$-chain complex $C_*$ with a choice of a preferred basis, we have
\[\torsionag(C_*)=\sum_{p=0}^{m-1} (-1)^p \det\nolimits_D(S_p(\gamma)) \in \Dab/\{\pm 1\}.\]
Furthermore, any $D$-acyclic finite free $\ZZ G$-chain complex with a choice of a preferred basis admits a non-degenerate $\tau$-chain.
\end{theo}
In the following, we will use the formula in \cref{theo:twisted:matrix_chain} as the definition of the agrarian torsion $\torsionag(C_*)$.

If $X$ is a finite free G-CW-complex that is $D$-acyclic, then its cellular $\ZZ G$-chain complex $C_*(X)$ will be a $D$-acyclic finite free $\ZZ G$-chain complex.
Up to orientation and the choice of representatives for the free $G$-orbits, the cell structure of $X$ determines a preferred choice of a basis for $C_*(X)$.
This observation leads to the following notion of agrarian torsion for G-CW-complexes:
\begin{defi}
Let $X$ be a $D$-acyclic finite free G-CW-complex. The \emph{$D$-agrarian torsion} of $X$ is defined as
\[\torsionag(X)\coloneqq \torsionag(C_*(X)) \in \Dab/\{\pm g\mid g\in G\},\]
where $C_*(X)$ is endowed with any $\ZZ G$-basis that projects to a $\ZZ$-basis of $C_*(X/G)$ consisting of unequivariant cells.
\end{defi}

\subsection{Agrarian Polytope}
Building on the notions of agrarian Betti numbers and agrarian torsion, we are now able to associate to a $D$-acyclic finite $G$-CW-complex $X$ a polytope.
This polytope, called the \emph{agrarian polytope} of $X$, arises as the convex hull of the support of the associated agrarian torsion, viewed as a quotient of suitable twisted polynomials.
The idea to study the Newton polytope of a torsion invariant goes back to~\cite{FL2017}, where the $L^2$-polytope of a certain subclass of all two-generator one-relator group is defined and used to prove the Friedl--Tillmann conjecture for them.

We begin with polytope-specific terminology:

\begin{defi}
Let $V$ be a finite-dimensional real vector space.
A \emph{polytope} in $V$ is the convex hull of finitely many points in $V$.
For a polytope $P\subset V$ and a linear map $\phi\colon V\to\RR$ we define
\[F_\phi(P)\coloneqq \{p\in P \mid \phi(p)=\min_{q\in P} \phi(q)\}\]
and call this polytope the \emph{$\phi$-face} of $P$.
The elements of the collection
\[\{F_\phi(P)\mid \phi\colon V\to\RR\}\]
 are the \emph{faces} of $P$.
 A face is called a \emph{vertex} if it consists of a single point.
\end{defi}

In the following, the ambient vector space $V$ will always be $\RR \otimes_{\ZZ} H$ for some finitely generated free abelian group $H$.
For such $V$, we will consider a special type of polytope:

\begin{defi}
A polytope $P$ in $V$ is called \emph{integral} if its vertices lie on the lattice $H\subset V$.
\end{defi}

Given two integral polytopes $P$ and $Q$ in $V$, their pointwise or \emph{Minkowski sum} $P + Q=\{p+q\mid p\in P, q\in Q\}$ is again an integral polytope.
Any vertex of the resulting polytope is a pointwise sums of a vertex of $P$ and a vertex of $Q$.
Equipped with the Minkowski sum the set of all integral polytopes in $V$ becomes a cancellative abelian monoid with neutral element $\{0\}$, see~\cite[Lemma~2]{Radstrom1952}.
Hence, the monoid embeds into its Grothendieck group, which was first considered in~\cite[6.3]{FT2015}:
\begin{defi}
Let $H$ be a finitely generated free abelian group.
Denote by $\Poly(H)$ the \emph{polytope group} of $H$, that is the Grothendieck group of the cancellative abelian monoid given by all integral polytopes in $\RR \otimes_{\ZZ} H$ under Minkowski sum.
In other words, let $\Poly(H)$ be the abelian group with generators the formal differences $P-Q$ of integral polytopes and relations $(P-Q)+(P'-Q')=(P+P')-(Q-Q')$ as well as $P-Q=P'-Q'$ if $P+Q'=P'+ Q'$.
The neutral element is given by the one-point polytope $\{0\}$, which we will drop from the notation.
We view $H$ as a subgroup of $\Poly(H)$ via the map $h\mapsto \{h\}$.
\end{defi}

An element of the polytope group that is of the form $P-0$, for which we also just write $P$, is called a \emph{single polytope} and is uniquely represented in this form. Any other element is called a \emph{virtual polytope}.

In order to later get well-defined invariants with values in the polytope group, we will mostly be dealing with the following quotient of the full polytope group:
\begin{defi}
The \emph{translation-invariant polytope group} of $H$, denoted by $\PolyT(H)$, is defined to be the quotient group $\Poly(H)/H$.
\end{defi}

The following simple construction underlies the definition of the $L^2$-polytope in \cite{FL2017} and will also be used to define the agrarian polytope:
\begin{defi}
Let $D$ be a skew field and let $H$ be a finitely generated free abelian group. Let $DH$ denote some twisted group ring formed out of $D$ and $H$. The \emph{Newton polytope} $P(p)$ of an element $p=\sum_{h\in H} u_h \ast h\in DH$ is the convex hull of the \emph{support} $\supp(p)=\{h\in H\mid u_h\neq 0\}$ in $\RR \otimes_{\ZZ} H$.
\end{defi}

Since $H$ is finitely generated free abelian, we can consider the Ore field of fractions $\Ore(DH)$ of the twisted group ring $DH$, just as we did in \cref{lemm:agrarian:ore_field}.
The definition of the Newton polytope can be extended to elements of $\Ore(DH)$ in the following way:
\begin{defi}
The group homomorphism
\begin{align*}
    P\colon \Dab[\Ore(DH)] &\to \Poly(H)\\
    pq^{-1}&\mapsto P(p)-P(q)
\end{align*}
is called the \emph{polytope homomorphism} of $\Ore(DH)$. It induces a homomorphism
\[P\colon \Dab[\Ore(DH)]/\{\pm h\mid h\in H\} \to \PolyT(H).\]
\end{defi}
It is easily verified in~\cite[Lemma~3.12]{Kielak2018} (and the discussion following the lemma) that $P$ is a well-defined group homomorphisms.

We now consider a finitely generated agrarian group $G$ and denote the free part of its abelianisation by $H$.
Let $K$ be the kernel of the projection of $G$ onto $H$.
In~\cite{FL2017}, assuming that the group $G$ satisfies the Atiyah conjecture, the polytope homomorphism is used for the Linnell skew field $\linnelld(G)$, which can conveniently be expressed as an Ore localisation of the twisted group ring $\linnelld(K)H$.
While the target of an arbitrary agrarian embedding $\alpha\colon\ZZ G\hookrightarrow D$ is not necessarily an Ore localisation of a suitable twisted group ring, this is true for its rationalisation, which we introduced in \cref{defi:agrarian:rationalisation}.

\begin{defi}
\label{defi:polytope:agrarian:chain_complex}
Let $\ZZ G\hookrightarrow D$ be an agrarian embedding for $G$ with rationalisation $\ZZ G\to D_r$.
Let $C_*$ be a $D_r$-acyclic finite based free $\ZZ G$-chain complex $C_*$.
The \emph{($D_r$-)agrarian polytope} of $C_*$ is defined as
\[\Pag{D_r}(C_*) \coloneqq P(-\torsionag[D_r](C_*)) \in \Poly(H),\]
where we use the polytope homomorphism associated to the skew field $D_r=\Ore(DH)$.
\end{defi}

The sign in the definition of the $D_r$-agrarian polytope is a matter of convention, but is chosen such that we get a single polytope in many cases of interest.
It is a consequence of \cite[Lemma~2.5]{HK2019} that the agrarian polytope does not depend on the particular choice of structure functions involved in the construction of the twisted group ring $DH$.

In the following, we will always consider the agrarian polytopes associated to cellular chain complexes of $G$-CW-complexes, where we have to account for the indeterminacy caused by choosing a suitable basis made of cells.
Since the $D_r$-agrarian torsion of a $G$-CW-complex naturally lives in $(D_r^\times)_{\textrm{ab}}/\{\pm g\mid g\in G\}$, the associated polytope will only be defined up to translation.

\begin{defi}
\label{defi:polytope:agrarian}
Let $\ZZ G\hookrightarrow D$ be an agrarian embedding for $G$ with rationalisation $\ZZ G\to D_r$.
Let $X$ be a $D_r$-acyclic finite free $G$-CW-complex.
The \emph{($D_r$-)agrarian polytope} of $X$ is defined as
\[\Pag{D_r}(X) \coloneqq \Pag{D_r}(C_*(X)) \in \PolyT(H).\]
\end{defi}

The property of the agrarian polytope that enables our applications is that it is a $G$-homotopy invariant:
\begin{prop}[{\cite[Proposition~5.8]{HK2019}}]
\label{prop:polytope:invariance}
The $D_r$-agrarian polytope $\Pag{D_r}(X)$ is a $G$-homotopy invariant of $X$.
\end{prop}
As a consequence, the $D_r$-agrarian polytope $\Pag{D_r}(X)$ does not depend on the particular $G$-CW-structure of $X$.

\subsection{Thickness of Newton polytopes} The agrarian polytope is usually rather difficult to compute for a concrete group. Its thickness along a given line is often more accessible. With an approach similar to~\cite{FL2016}, we will see in \cref{sect:twisted} that it can be computed in terms of agrarian Betti numbers of a suitably restricted chain complex.
\begin{defi}
\label{thickness dfn}
Assume that $G$ is finitely generated and denote the free part of its abelianisation by $H$. Let $\phi\colon G\to \ZZ$ be a homomorphism factoring through $H$ as $\overbar\phi \colon H\to\ZZ$. Let $P\in\Poly(H)$ be a single polytope. The \emph{thickness} of $P$ along $\phi$ is given by
\[\thickness_\phi(P)\coloneqq \max\{\overbar\phi(x) - \overbar\phi(y)\mid x, y\in P\}\in \ZZ_{\geq 0}.\]
Since it respects the Minkowski sum and vanishes on polytopes consisting of a single point, the assignment $P\mapsto \thickness_\phi(P)$ extends to a group homomorphism $\thickness_\phi\colon\PolyT(H)\to\ZZ$.
\end{defi}

An equivalent way of thinking of a twisted group ring $DH$ constructed from an agrarian embedding $\ZZ G\hookrightarrow D$ in the case $H=\ZZ$ is as a twisted Laurent polynomial ring $D[t, t^{-1}]$. In order to see the correspondence, note that since $\ZZ$ is free with one generator, we can choose a section $s$ of the epimorphism $\phi\colon G\to \ZZ$ which is itself a homomorphism. By \cref{lemm:agrarian:ore_field}, the resulting twisted group ring will be independent of the choice of the (group-theoretic or not) section.
If we stipulate that $tdt^{-1}=s(1)ds(1)^{-1}$ for $d \in D$, then the ring $D[t, t^{-1}]_\phi$, with  $\phi$ added as an index to indicate the origin of the twisting, will be canonically isomorphic to $D\ZZ$.

For elements of the Laurent polynomial ring, the Newton polytope will be a line of length equal to the degree of the polynomial. Here, the \emph{degree} $\deg(x)$ of a non-trivial Laurent polynomial $x$ is the difference of the highest and lowest degree among its monomials. In particular, the degree of a single monomial is always 0 and the degree of a polynomial with non-vanishing constant term coincides with its degree as a Laurent polynomial.

Let now $G$ be a finitely generated agrarian group with agrarian embedding $\ZZ G\hookrightarrow D$ and denote by $K$ the kernel of the projection of $G$ onto the free part of its abelianisation, which we denote $H$. Further let $\phi\colon G\to\ZZ$ be an epimorphism with kernel $K_\phi$, and denote the induced map $H\to\ZZ$ by $\overbar\phi$ with kernel $K_{\overbar\phi}$. Recall that by \cref{lemm:agrarian:ore_comparison}, the iterated Ore field $\Ore(\Ore(DK_{\overbar\phi})\ZZ)$ can be identified with the Ore field $\Ore(DH)$ via the isomorphism $\beta$. We write $\Ore(DK_{\overbar\phi})\ZZ$ as a twisted Laurent polynomial ring $\Ore(DK_{\overbar\phi})[t,t^{-1}]_\phi$.
The idea behind the following lemma is now based on the fact that the Newton polytope of a multi-variable Laurent polynomial $x$ determines all the Newton `lines' of $x$ when viewed as a single-variable Laurent polynomial with more complicated coefficients.
\begin{lemm}
\label{lemm:polytope:one_variable}
In the situation above, for any $x\in\Ore(DK_{\overbar\phi})[t, t^{-1}]_\phi$ with $x\neq 0$, we have
\[\thickness_{\phi}(P(\beta(x)))=\deg(x).\]
\end{lemm}
\begin{proof}
Since multiplying by a common denominator of all $\Ore(DK_{\overbar\phi})$-coefficients of $x$ does neither change its degree nor the support of its image under $\beta$, we can restrict to the case $x\in DK_{\overbar\phi}[t,t^{-1}]_\phi$.
Thus $x$ will be of the form $x=\sum_{n\in\ZZ} (\sum_{k\in K_{\overbar\phi}} u_{k, n}\ast k) t^n$ with $u_{k, n}\in D$. Denoting the group-theoretic section of $\overbar\phi$ used to construct the twisted Laurent polynomial ring by $s$, we obtain:
\[\beta(x)= \sum_{\substack{\mathllap{n}\in \mathrlap{\ZZ} \\ \mathllap{k}\in \mathrlap{K_{\overbar\phi}}}} u_{k,n} \ast k s(n).\]
The elements $ks(n)$ form a basis of the free $D$-module $DH$, and thus no cancellation can occur between the individual $u_{k,n}$. By the analogous argument for the twisted group ring $DK_{\overbar\phi}$, cancellation can also be ruled out for the sum $\sum_{k\in K_{\overbar\phi}} u_{k, n}\ast k$ for each $n\in \ZZ$. We conclude:
\begin{align*}
\thickness_{\phi}(P(\beta(x))) =&\max\{\overbar\phi(k_1s(n_1)) - \overbar\phi(k_2s(n_2))\mid k_1, k_2 \in K_{\overbar\phi}, n_1, n_2\in \ZZ, u_{k_i, n_i}\neq 0\}\\
=&\max\{n_1 - n_2\mid k_1, k_2 \in K_{\overbar\phi}, n_1, n_2\in \ZZ, u_{k_i, n_i}\neq 0 \}\\
=&\max\{n_1 - n_2\mid \exists k_i\in K_{\overbar\phi}\colon u_{k_i,n_i}\neq 0 \textrm{ for } i=1,2\}\\
=&\max\{n_1 - n_2\mid \sum_{k_i\in K_{\overbar\phi}} u_{k_i, n_i}\ast k_i\neq 0 \textrm{ for } i=1,2\}\\
=&\deg(x).\qedhere
\end{align*}
\end{proof}

\section{Twisted agrarian Euler characteristic}
\label{sect:twisted}

While the shape of the agrarian polytope introduced in the previous section is often hard to determine, there is a convenient equivalent description of its thickness along a given line. To this end, we will introduce the agrarian analogue of the twisted $L^2$-Euler characteristic introduced by Friedl and Lück in~\cite{FL2016}. We assume that $G$ is a finitely generated $D$-agrarian group with a fixed agrarian embedding $\alpha\colon \ZZ G\hookrightarrow D$. We use $H$ to denote the free part of the abelianisation of $G$, and let $K$ be the kernel of the canonical projection of $G$ onto $H$.

\subsection{Definition of the twisted agrarian Euler characteristic}

We now introduce twisted agrarian Euler characteristics, which arise as ordinary agrarian Euler characteristics of cellular $\ZZ G$-chain complexes twisted by an epimorphism from $G$ to the integers:
\begin{defi}
Let $X$ be a finite free $G$-CW-complex and let $\phi\colon G\to\ZZ$ be a homomorphism. We denote by $\phi^*\ZZ[t,t^{-1}]$ the $\ZZ G$-module obtained from the $\ZZ$-module $\ZZ[t,t^{-1}]$ by letting $G$ act as $g \cdot \sum_{n \in \ZZ} \lambda_n t^n= \sum_{n\in \ZZ} \lambda_n t^{n+\phi(g)}$, where $\lambda_n\in\ZZ$ for $n\in\ZZ$.
Consider the $\ZZ G$-chain complex $C_*(X)\otimes_{\ZZ} \phi^*\ZZ[t,t^{-1}]$ equipped with the diagonal $G$-action and set
\begin{align*}
\bag_p(X;\phi) &\coloneqq \bag_p(C_*(X)\otimes_{\ZZ} \phi^*\ZZ[t,t^{-1}])\in \NN \cup \{\infty\},\\
\hag(X;\phi) &\coloneqq \sum_{p\geq 0} \bag_p(X;\phi)\in\NN \cup \{\infty\},\\
\chiag(X;\phi) &\coloneqq \sum_{p\geq 0} (-1)^p \bag_p(X;\phi)\in \ZZ,\text{ if $\hag(X;\phi)<\infty$}.
\end{align*}
We say that $X$ is \emph{$\phi$-$D$-finite} if $\hag(X;\phi)<\infty$, and in this case $\chiag(X;\phi)$ is called the \emph{$\phi$-twisted $D$-agrarian Euler characteristic} of $X$. More generally, we will also consider the $\phi$-twisted agrarian Euler characteristic $\chiag(C_*;\phi)$ for any finite free $\ZZ G$-chain complex $C_*$, with $C_*$ taking the role of the cellular chain complex $C_*(X)$.
\end{defi}

The aim of this section is to prove that the thickness of the agrarian polytope in a prescribed direction can be computed as a twisted agrarian Euler characteristic.
Recall that $G$ is a finitely generated $D$-agrarian group with a fixed agrarian embedding $\alpha\colon\ZZ G\hookrightarrow D$ and that we denote by $\alpha_r\colon\ZZ G\hookrightarrow D_r$ the rationalisation of $\alpha$ as introduced in~\cref{defi:agrarian:rationalisation}.
\begin{theo}
\label{theo:twisted:thickness}
Let $X$ be a $D_r$-acyclic finite free $G$-CW-complex and $\phi\colon G\to\ZZ$ a homomorphism.
Then
\[\thickness_\phi(\Pag{D_r}(X))=-\chiag[D_r](X;\phi).\]
\end{theo}

For universal $L^2$-torsion, the analogous statement has been proved by Friedl and Lück in~\cite[Remark~4.30]{FL2017}. Their proof is based on the fact that universal $L^2$-torsion is the universal abelian invariant of $L^2$-acyclic finite based free $\ZZ G$-chain complexes $C_*$ that is additive on short exact sequences and satisfies a certain normalisation condition. While large parts of the verification of this universal property are purely formal, in the proof of~\cite[Lemma~1.5]{FL2017} it is used that the combinatorial Laplace operator on $C_*$ induces the $L^2$-Laplace operator on $\mathcal{N}(G)\otimes C_*$, which has no analogue over a general skew field $D$. We instead establish \cref{theo:twisted:thickness} using the matrix chain approach to the computation of Reidemeister torsion explained in~\cite[I.2.1]{Turaev2001}.

\subsection{Reduction to ordinary Euler characteristics} Before we get to the proof, we will transfer some of the helpful lemmata in~\cite[Sections~2.2~\&~3.3]{FL2016} to the agrarian setting.

The following lemma allows us to restrict our attention to surjective twists $\phi\colon G\to\ZZ$ in the proof of \cref{theo:twisted:thickness}:
\begin{lemm}
\label{lemm:twisted:surjective}
Let $X$ be a finite free $G$-CW-complex and let $\phi\colon G\to\ZZ$ be a group homomorphism.
\begin{enumerate}[label=(\arabic*)]
    \item For any integer $k\geq 1$ we have that $X$ is $(k\cdot\phi)$-$D$-finite if and only if $X$ is $\phi$-$D$-finite, and if this is the case we get
    \[\chiag(X;k\cdot\phi)=k\cdot\chiag(X;\phi).\]
    \item Denote the trivial homomorphism $G\to\ZZ$ by $c_0$. The complex $X$ is $c_0$-$D$-finite if and only if $X$ is $D$-acylic, and if this is the case we get
    \[\chiag(X;c_0)=0.\]
\end{enumerate}
\end{lemm}
\begin{proof}
\begin{enumerate}[label=(\arabic*)]
\item This follows from the direct sum decomposition $(k\cdot \phi)^* \ZZ[t,t^{-1}]\cong \bigoplus_{i=1}^k \phi^*\ZZ[t,t^{-1}]$ and additivity of Betti numbers.
\item This is a direct consequence of $C_*(X)\otimes_{\ZZ} c_0^*\ZZ[t,t^{-1}]\cong \bigoplus_{\ZZ} C_*(X)$ and additivity of Betti numbers.\qedhere
\end{enumerate}
\end{proof}

We will now see that twisted $D$-agrarian Euler characteristics over $G$ can equivalently be viewed as ordinary $D$-agrarian Euler characteristics over the kernel of the twist homomorphism.
\begin{lemm}
\label{lemm:twisted:kernel}
Let $X$ be a finite free $G$-CW-complex and let $\phi\colon G\to\ZZ$ be an epimorphism. Denote the kernel of $\phi$ by $K_\phi$. Then $X$ is $\phi$-$D$-finite if and only if $\sum_{p\geq 0} \bag_p(\res_G^{K_\phi}X)<\infty$, and in this case we have
\[\chiag(X;\phi)=\chiag(\res_G^{K_\phi}X).\]
\end{lemm}
\begin{proof}
The proof is based on the following isomorphism of $\ZZ G$-chain complexes:
\begin{align*}
\ZZ G\otimes_{\ZZ K_\phi} \res_G^{K_\phi}C_*(X)&\xrightarrow{\cong}C_*(X)\otimes_{\ZZ}\phi^* \ZZ[t,t^{-1}]\\
g\otimes x&\xrightarrow{\phantom{\cong}} gx\otimes t^{\phi(g)},
\end{align*}
the inverse of which is given by $y\otimes t^q\mapsto g\otimes g^{-1}y$ for any choice of $g\in \phi^{-1}(q)$.
Using the isomorphism, we obtain for every $p\geq 0$:
\begin{align*}
H_p(D\otimes C_*(X)\otimes_{\ZZ} \phi^*\ZZ[t,t^{-1}]) &\cong H_p(D\otimes \ZZ G\otimes_{\ZZ K_\phi} \res_G^{K_\phi} C_*(X))\\
&=H_p(D\otimes_{\ZZ K_\phi} \res_G^{K_\phi} C_*(X)).
\end{align*}
We conclude that $\bag_p(X;\phi)=\bag_p(\res_G^{K_\phi} X)$ by applying $\dim_D$, which yields the claim after taking the alternating sum over $p\geq 0$.
\end{proof}

\begin{rema}
\label{rema:twisted:typef}
Let $G$ be a $D$-agrarian group of type $\mathtt{F}$. Let $\phi\colon G\to\ZZ$ be an epimorphism with kernel $K_\phi$. If $K_\phi$ is also of type $\mathtt{F}$, then by \cref{lemm:twisted:kernel} and \cite[Theorem~3.9~(2)]{HK2019}
\[\chiag(EG;\phi)=\chiag(\res_G^{K_\phi} EG)=\chiag(EK_\phi)=\chi(K_\phi).\]
In particular, in this case the value of $\chiag(EG;\phi)$ does not depend on the choice of agrarian embedding.
\end{rema}

\begin{lemm}
\label{lemm:twisted:laurent_computation}
Let $C_*$ be a $D$-acyclic $\ZZ G$-chain complex of finite type. Let $\phi\colon G\to\ZZ$ be an epimorphism with kernel $K_\phi$. Consider the embedding $\ZZ G\cong (\ZZ K_\phi)\ZZ \hookrightarrow D\ZZ=D[t,t^{-1}]_\phi$ constructed in \cref{lemm:agrarian:ore_field} for $K\coloneqq K_\phi$, where we use that $G/K\cong\ZZ$ via $\phi$. Then
\[\bag_n(\res_G^{K_\phi} C_*)= \dim_D H_n(D[t,t^{-1}]_\phi\otimes C_*)<\infty.\]
In particular, the $D[t, t^{-1}]_\phi$-modules $H_n(D[t,t^{-1}]_\phi\otimes C_*)$ are torsion.
\end{lemm}
\begin{proof}
The proof is analogous to that of~\cite[Theorem~3.6~(4)]{FL2016} with $D$ taking the role of $\linnelld(K)$. The assumption that $C_*$ be projective is in fact not used in the proof of the theorem and hence is not part of the statement of \cref{lemm:twisted:laurent_computation}.
\end{proof}

\begin{coro}
\label{coro:twisted:laurent_computation}
Let $X$ be a $D$-acyclic finite free $G$-CW-complex. Let $\phi\colon G\to\ZZ$ be an epimorphism with kernel $K_\phi$. Then $X$ is $\phi$-$D$-finite and
\[\chiag(X, \phi)=\sum_{p\geq 0} (-1)^p \dim_D H_p(D[t, t^{-1}]_\phi\otimes C_*(X)).\]
\end{coro}
\begin{proof}
Apply \cref{lemm:twisted:kernel,lemm:twisted:laurent_computation}.
\end{proof}

\subsection{Thickness of the agrarian polytope}
We are now able to proceed with the proof of \cref{theo:twisted:thickness}:
\begin{proof}[Proof of \cref{theo:twisted:thickness}]
We will actually prove the more general statement that for every $D_r$-acyclic finite based free $\ZZ G$-chain complex $C_*$ concentrated in degrees 0 through $m$
\begin{equation}
    \label{eq:twisted:to_prove}
    \thickness_{\phi}(P(-\torsionag[D_r](C_*))=-\chiag[D_r](C_*;\phi).
\end{equation}
Since $\thickness_{\phi}$ and $P$ are homomorphisms, we can drop the signs from both sides. Using \cref{lemm:twisted:surjective}, we can further assume that $\phi$ is an epimorphism.

By \cref{theo:twisted:matrix_chain}, we find a non-degenerate $\tau$-chain $\gamma$ such that
\[\thickness_{\phi}\Big(P\big(\torsionag[D_r](C_*)\big)\Big)=\thickness_{\phi}\Big(P\Big(\sum_{p=0}^m (-1)^{p} \det\nolimits_{D_r}\big(S_p(\gamma)\big)\Big)\Big).\]
Crucially,
\[\Ore(\Ore(DK_{\overbar\phi})[t,t^{-1}]_\phi)\cong \Ore(DH)=D_r\]
via the isomorphism $\beta$ constructed in \cref{lemm:agrarian:ore_comparison}, where $K_{\overbar\phi}$ is the kernel of the epimorphism $\overbar\phi\colon H\to \ZZ$ induced by $\phi$.
The subring
\[\Ore(DK_{\overbar\phi})[t,t^{-1}]_\phi\]
of the left-hand side, which contains $\beta^{-1}(\ZZ G)$ and thus all entries of  $S_p=S_p(\gamma)$, is a (non-commutative) Euclidean domain. This means that we can diagonalise the matrices $S_p$ by multiplying them from the left and right with permutation matrices and elementary matrices over this twisted Laurent polynomial ring. This diagonalisation procedure occurs as part of an algorithm that brings a matrix into Jacobson normal form, which is a non-commutative analogue of the better-known Smith normal form for matrices over commutative PIDs. For details, we refer to the proof of \cite[Theorem~3.10]{Jacobson1943}. Recall that a \emph{permutation matrix} is a matrix obtained from an identity matrix by permuting rows and columns. An \emph{elementary matrix} over a ring $R$ is a matrix differing from the identity matrix in a single off-diagonal entry. The determinant of either type of matrix is $1$ or $-1$, and thus the thickness in direction of $\phi$ of their polytopes vanish. Hence, $\thickness_{\phi}(P(\det(S_p)))=\thickness_{\phi}(P(\det(T_p)))$ for the diagonal matrix $T_p$ obtained from $S_p$ in this way. We denote the diagonal entries of $T_p$ by $\lambda_{p, i}\in\Ore(DK_{\overbar\phi})[t, t^{-1}]_\phi$ for $i=1,\dots, |\gamma_p|$ and note that all the entries $\lambda_{p, i}$ are non-zero since all matrices $S_p$ become invertible over $D_r$. Using that both $\thickness_{\phi}$ and $P$ are homomorphisms, and applying \cref{lemm:polytope:one_variable} once more, we compute:
\begin{align*}
\thickness_{\phi}(P(\torsionag[D_r](C_*)) &= \thickness_{\phi}\Big(P\Big(\sum_{p=1}^m (-1)^p \det\nolimits_{D_r}\big(S_p(\gamma)\big)\Big)\Big)\\
&= \sum_{p=0}^{m-1} (-1)^{p} \sum_{i=1}^{|\gamma_p|}\thickness_{\phi}(P(\beta(\lambda_{p, i})))\\
&= \sum_{p=0}^{m-1} (-1)^{p} \sum_{i=1}^{|\gamma_p|} \deg(\lambda_{p, i}).
\end{align*}

We will now consider the right-hand side of \eqref{eq:twisted:to_prove}.
For this, we use that the agrarian embedding $\ZZ K_\phi\hookrightarrow D_r=\Ore(\Ore(DK_{\overbar\phi})\ZZ)$ factors through the agrarian embedding $\ZZ K_\phi\hookrightarrow\Ore(DK_{\overbar\phi})$, and thus the embedding $\ZZ G\cong(\ZZ K_\phi)\ZZ\hookrightarrow D_r[t, t^{-1}]_\phi$ introduced in \cref{lemm:twisted:laurent_computation} factors through $\ZZ G\cong (\ZZ K_\phi)\ZZ\hookrightarrow\Ore(DK_{\overbar\phi})[t, t^{-1}]_\phi$.
Since $D_r$ is flat over the skew field $\Ore(DK_{\overbar\phi})$, we conclude from \cref{coro:twisted:laurent_computation} that
\begin{align*}
\chiag[D_r](C_*; \phi)&=\sum_{p=0}^m (-1)^p \dim_{D_r} H_p(D_r[t, t^{-1}]_\phi\otimes C_*)\\
&=\sum_{p=0}^m (-1)^p \dim_{\Ore(DK_{\overbar\phi})} H_p(\Ore(DK_{\overbar\phi})[t, t^{-1}]_\phi\otimes C_*).
\end{align*}
Since $C_*$ is $D_r$-acyclic, we have $H_m(D_r\otimes C_*)=0$. But $C_{m+1}$ is trivial, which means that the differential $c_m$ must be injective. In particular, the summand corresponding to $p=m$ vanishes.

In order to establish \eqref{eq:twisted:to_prove}, we are now left to prove that
\begin{equation}
\label{eq:twisted:claim}
\sum_{i=1}^{|\gamma_p|} \deg(\lambda_{p, i})=\dim_{\Ore(DK_{\overbar\phi})} H_p(\Ore(DK_{\overbar\phi})[t,t^{-1}]_\phi\otimes C_*)
\end{equation}
holds for $p=0,\ldots,m-1$.
In order to not overload notation, we abbreviate $\Ore(DK_{\overbar\phi})[t,t^{-1}]_\phi$ as $R$.
Recall that the homology modules $H_p(R\otimes C_*)$ consist solely of $R$-torsion elements by \cref{lemm:twisted:laurent_computation}.
Furthermore, since $R\otimes C_{p-1}$ is a free $R$-module, any $R$-torsion maps into it trivially. We are thus able to express the homology modules as torsion submodules of a cokernel in the following way:
\begin{align*}
H_p(R\otimes C_*) &= \ker(\id_R\otimes c_p)/\im(\id_R\otimes c_{p+1})\\
&\cong \ker\left(\id_R\otimes c_p\colon (R\otimes C_p) / \im(\id_R\otimes c_{p+1}) \to R\otimes C_{p-1}\right)\\
&=\tors_R((R\otimes C_p)/\im(\id_R\otimes c_{p+1}))\\
&=\tors_R(\coker(\id_R\otimes c_{p+1})).
\end{align*}
Instead of performing elementary operations on the matrix $S_p$ to obtain the diagonal matrix $T_p$, we can instead apply them to the entire matrix $A_p$ representing $\id_R\otimes c_{p+1}$. This procedure will not change the isomorphism type of the cokernel of the map given by right multiplication with this matrix. Applying further elementary operations over $R$, we can achieve that all the entries not contained in $S_p$ consist only of zeros with the submatrix $S_p$ now being of the form $T_p$. This is possible since $S_p$ has the same rank as $A_p$ over the field of fractions of $\Ore(R)$ by the same rank counting argument used to prove~\cite[I.2.2]{Turaev2001}.
Hence
\[H_p(R\otimes C_*)\cong\tors_R(\coker(\id_R\otimes c_{p+1}))\cong\oplus_{i=1}^{|\gamma_{p}|} R/(\lambda_{p,i}),\]
which yields \eqref{eq:twisted:claim} after applying $\dim_{\Ore(DK_{\overbar\phi})}$.
\end{proof}

\section{The Bieri--Neumann--Strebel invariants and HNN extensions} In order to discuss some application of the theory of agrarian invariants, we need to first cover the BNS invariants and the HNN extensions.

\begin{defi}
Let $G$ be a group generated by a finite subset $S$, and let $X$ denote the Cayley graph of $G$ with respect to $S$. Recall that the vertex set of $X$ coincides with $G$. We define the \emph{Bieri--Neumann--Strebel} (or \emph{BNS}) \emph{invariant} $\Sigma^1(G)$ to be the subset of $H^1(G;\RR) \smallsetminus \{0\}$ consisting of the non-trivial homomorphisms (the \emph{characters}) $\phi \colon G \to \RR$ for which the full subgraph of $X$ spanned by $\phi^{-1}([0,\infty))\subseteq G$ is connected.
\end{defi}

The BNS invariants were introduced by Bieri, Neumann and Strebel in~\cite{Bierietal1987} via a different, but equivalent definition. It is an easy exercise to see that $\Sigma^1(G)$ is independent of the choice of the finite generating set $S$.

We now aim to give an interpretation of lying in the BNS invariant for integral characters $\phi \colon G \to \ZZ$. To do so, we need to introduce the notion of HNN extensions.

\begin{defi}
\label{defi:bns:bns}
Let $A$ be a group and let $\alpha\colon B\xrightarrow{\cong} C$ be an isomorphism between two subgroups of $A$. Choose a presentation $\langle S\mid R\rangle$ of $A$ and let $t$ be a new symbol not in $S$. Then the group $A*_\alpha$ defined by the presentation
\[\langle S, t\mid R, tbt^{-1}=\alpha(b)\ \forall b\in B\rangle\]
is called the \emph{HNN extension of $A$ relative to $\alpha\colon B \xrightarrow{\cong}C$}. We call $A$ the \emph{base group} and $B$ the \emph{associated group} of the HNN extension.

The HNN extension is called \emph{ascending} if $B = A$.

The homomorphism $\phi \colon A*_\alpha \to \ZZ$ given by $\phi(t) = 1$ and $\phi(s) = 0$ for every $s \in S$ is the \emph{induced character}.
\end{defi}

\begin{prop}[{\cite[Proposition 4.3]{Bierietal1987}}]
\label{prop:bns:hnn}
Let $G$ be a finitely generated group, and let $\phi \colon G \to \ZZ$ be a non-trivial character. We have $\phi \in \Sigma^1(G)$ if and only if $G$ is isomorphic to an ascending HNN extension with finitely generated base group and induced character $\phi$.
\end{prop}

\begin{defi}
\label{marked def}
Suppose that $G$ is finitely generated.
Let $P$ be a single polytope in the $\RR$-vector space $H_1(G;\RR)$, and let $F$ be a face of $P$. A \emph{dual} of $F$ is a connected component of the subspace
\[
\{ \phi \in H^1(G;\RR)\setminus\{0\} \mid F_\phi(P) = F \}.
\]

A \emph{marked polytope} is a pair $(P,m)$, where $P$ is a single polytope in $H_1(G;\RR)$, and $m$ is a \emph{marking}, that is a function $m \colon H^1(G;\RR) \to \{0,1\}$, which is constant on duals of faces of $F$, and such that $m^{-1}(1)$ is open.

The pair $(P,m)$ is a \emph{polytope with marked vertices} if $m^{-1}(1)$ is a union of some duals of vertices of $P$.

The marking $m$ will usually be implicit, and the characters $\phi$ with $m(\phi) = 1$ will be called \emph{marked}.
\end{defi}

In~\cite{FT2015}, Friedl--Tillmann use a different notion of a marking of a polytope, which corresponds to a polytope with marked vertices in our terminology where the marking $m$ is additionally required to be constant on all duals of a given vertex. Thus, our notion is more general, and the two notions differ when the polytope in question is a singleton in a $1$-dimensional ambient space: with our definition of marking, such a polytope admits four distinct markings (just as every compact interval of non-zero length does), whereas with the Friedl--Tillmann definition such a polytope admits only two markings in which either every character is marked or none is.

\section{Application to two-generator one-relator groups}
\begin{defi}
    A \emph{$(2,1)$-presentation} is a group presentation of the form ${\langle x, y \mid r\rangle}$, i.e., with two generators and a single relator.
    A group that admits a $(2, 1)$-presentation is called a \emph{two-generator one-relator group}.
\end{defi}

The story of the usefulness of agrarian invariants for two-generator one-relator groups begins with the following result of Lewin--Lewin.
\begin{theo}[{\cite[Theorem~1]{LL1978}}]
\label{theo:tgor:agrarian}
Torsion-free one-relator groups are agrarian.
\end{theo}

In the following, for a group presentation $\pi$, we will denote the groups it presents by $G_\pi$.

In order to describe the cellular chain complex of the universal coverings of classifying spaces for two-generator one-relator groups, we will use Fox derivatives, which were originally defined in~\cite{Fox1953}. Let $F$ be a free group on generators $x_i, i\in I$. The \emph{Fox derivative with respect to $x_i$} is then defined to be the unique $\ZZ$-linear map $\frac{\partial}{\partial x_i}\colon \ZZ F\to \ZZ F$ satisfying the conditions
\[\frac{\partial 1}{\partial x_i} = 0, \frac{\partial x_i}{\partial x_j} = \delta_{ij} \text{ and } \frac{\partial uv}{\partial x_i} = \frac{\partial u}{\partial x_i} + u\frac{\partial v}{\partial x_i}\]
for all $u, w\in F$, where $\delta_{ij}$ denotes the Kronecker delta.
The fundamental formula for Fox derivatives~\cite[(2.3)]{Fox1953} states that for every $u\in \ZZ F$ we have
\[u-1=\sum_{i\in I} \frac{\partial u}{\partial x_i}\cdot (x_i-1).\]
In the particular case of a two-generator one-relator group $G=\langle x, y\mid r\rangle$, the fundamental formula applied to $r$ implies that the following identity holds in $\ZZ G$, since there $r-1=0$:
\begin{equation}
    \label{eq:tgor:fundamental}
    \frac{\partial r}{\partial x} \cdot (x-1) = - \frac{\partial r}{\partial y}\cdot (y-1).
\end{equation}

We will need the following non-triviality result for Fox derivatives in two-generator one-relator groups:

\begin{lemm}
    \label{lemm:tgor:nontriviality}
    Let $\pi=\langle x, y\mid r\rangle$ be a $(2, 1)$-presentation with cyclically reduced relator $r$, and take $z$ to denote either $x$ or $y$.
    Denote the number of times $z$ or $z^{-1}$ appears in the word $r$ by $s$.
    Then the Fox derivative $\partial r/\partial z\in \ZZ F_2$ is a sum of the form $\sum_{j=1}^{s} \pm w_j$ for words $w_j$ representing mutually distinct elements $g_j\in G_\pi$.
    In particular, $\partial r/\partial z\neq 0$ in $\ZZ G_\pi$ if $s>0$.
\end{lemm}
\begin{proof}
    This follows from \cite[Corollary~3.4]{FT2015}.
    While the statement of the corollary only asserts the distinctness of the group elements $g_j$ together with their scalar factors of $\pm 1$, the proof actually shows that the elements themselves are distinct.
    Also note that, in the proof of the corollary, $n_s$ is actually always strictly smaller than $l$, which is crucial for the correctness of the penultimate sentence.
\end{proof}

We are now able to show that the agrarian torsion of torsion-free two-generator one-relator groups is defined and can be calculated explicitly:
\begin{lemm}
    \label{lemm:tgor:torsion}
    Let $\pi=\langle x,y\mid r\rangle$ be a $(2, 1)$-presentation with $r$ cyclically reduced. Denote the universal covering of the presentation $2$-complex of $G_\pi$ associated to this presentation by $EG_\pi$. Then $EG_\pi$ is contractible and $D$-acyclic with respect to any agrarian embedding $\ZZ G_\pi\hookrightarrow D$. If $x$ or $x^{-1}$ appears as a letter in $r$, then
    \[\torsionag(EG_\pi)= -\left[\frac{\partial r}{\partial x}\right] + [y-1] \in \Dab,\]
    where $[-]\colon D^\times\to \Dab$ is the canonical quotient map.
    If $y$ or $y^{-1}$ appears in $r$, then the analogous statement holds with the roles of $x$ and $y$ interchanged.
\end{lemm}
\begin{proof}
That $EG_\pi$ is contractible follows from \cite[Chapter~III,~Proposition~11.1]{LS2001}. The cellular $\ZZ G_\pi$-chain complex of $EG_\pi$ takes the following form in terms of the Fox derivatives $\frac{\partial r}{\partial x}$ and $\frac{\partial r}{\partial y}$, see~\cite{Fox1953}:
\[\ZZ G_\pi \xrightarrow{\begin{pmatrix}\frac{\partial r}{\partial x} & \frac{\partial r}{\partial y}\end{pmatrix}} \ZZ G_\pi^2\xrightarrow{\begin{pmatrix}x-1 \\ y-1 \end{pmatrix}} \ZZ G_\pi.\]
We will now construct a non-degenerate $\tau$-chain for the associated $D$-chain complex and simultaneously obtain that the complex is acyclic.
Note that acyclicity is also a general consequence of the existence of a non-degenerate $\tau$-chain by~\cite[Lemma~I.2.5]{Turaev2001}.

Since $r$ is assumed to be cyclically reduced, the only case in which any of the generators is trivial in $\ZZ G_\pi$ is when $r$ consists of a single letter. Let us suppose for now that this is the case, and without loss of generality let us take $r = x$.
In this case, the chain complex under investigation becomes
\[\ZZ G_\pi \xrightarrow{\begin{pmatrix}1 & 0\end{pmatrix}} \ZZ G_\pi^2\xrightarrow{\begin{pmatrix}0 \\ y-1 \end{pmatrix}} \ZZ G_\pi.\]
Since $y \neq 1$ as $G=\langle y\rangle$, we immediately see that the complex is $D$-acyclic and comes with an obvious choice of a non-degenerate $\tau$-chain.

We will now assume that both generators represent non-trivial elements of $G_\pi$.
By~\cref{lemm:tgor:nontriviality}, the Fox derivative $\frac{\partial r}{\partial x}$ resp. $\frac{\partial r}{\partial y}$ represents the trivial element of $\ZZ G_\pi$ and hence of $D$ only if $x$ resp. $y$ does not appear in the word $r$, possibly inverted.
Since $G_\pi$ is not the free group on two generators, at least one of the letters $x$ and $y$ appears in this way, and hence at least one of the Fox derivatives represents an invertible element in $D$.

In conclusion, both differentials in $D\otimes C_*(EG_\pi)$ have maximal rank, namely $1$, and so the complex is acyclic, since it is a complex of modules over a skew field.

We obtain a non-degenerate $\tau$-chain by choosing the submatrices $S_1$ and $S_0$ to correspond to a non-trivial Fox derivative and the generator which is not the one with respect to which that Fox derivative was taken, respectively.
With this choice, the formula for the agrarian torsion is obtained from~\cref{theo:twisted:matrix_chain}.
\end{proof}

By the work of Waldhausen~\cite[Theorem~17.5 \& Theorem 19.4]{Waldhausen1978}, two presentation complexes associated to two $(2,1)$-presentations of isomorphic torsion-free two-generator one-relator groups are always simple homotopy equivalent. Since agrarian Betti numbers are homotopy invariant and agrarian torsion is a simple homotopy invariant by~\cite[Lemma~4.9]{HK2019}, \cref{lemm:tgor:torsion} actually implies that $EG$ is $D$-acyclic for every torsion-free two-generator one-relator group $G$ and its agrarian torsion can be calculated from any $(2, 1)$-presentation $\langle x, y\mid r\rangle$ with $r$ cyclically reduced.

Since the agrarian polytope is homotopy invariant by \cref{prop:polytope:invariance}, we obtain the following result even without appealing to the work of Waldhausen:

\begin{prop}
    \label{prop:tgor:polytope}
    Let $G$ be a torsion-free two-generator one-relator group that is not isomorphic to the free group on two generators, and let $\ZZ G\hookrightarrow D$ be an agrarian embedding.
    Denote the free part of the abelianisation of $G$ by $H$.
    If $\pi=\langle x, y\mid r\rangle$ is any $(2,1)$-presentation of $G$ such that $r$ is cyclically reduced and $x$ or $x^{-1}$ appears as a letter in $r$, we have
    \[\Pag{D_r}(G)=\Pag{D_r}(EG_\pi)=P([\partial r/\partial x])-P([y-1])\in \PolyT(H).\]
    If $y$ or $y^{-1}$ appears in $r$, then the analogous statement holds with the roles of $x$ and $y$ interchanged.
\end{prop}

Since the space $EG$ is unique up to $G$-homotopy equivalent, the polytope $\Pag{D_r}(G)$ is an invariant of the group $G$ and does not depend on the choice of a $(2,1)$-presentation.

\smallskip

In \cite{FT2015}, Friedl and Tillmann associate a polytope to \emph{nice} $(2, 1)$-presentations, which are defined as follows:

\begin{defi}
\label{nice def}
A $(2, 1)$-presentation $\pi=\langle x,y\mid r\rangle$ giving rise to a group $G_\pi$ is called \emph{nice} if
\begin{enumerate}[label={(\arabic*)}]
    \item $r$ is a non-empty word,
    \item $r$ is cyclically reduced and
    \item $b_1(G_\pi)=2$.
\end{enumerate}
\end{defi}

Their construction of the polytope is equivalent to the following definition by by~\cite[Proposition~3.5]{FT2015}:

\begin{defi}
\label{defi:tgor:polytope}
Let $\pi=\langle x, y \mid r\rangle$ be a nice $(2,1)$-presentation giving rise to a group $G_\pi$. Denote by $H$ the free part of the abelianisation of $G$ and write $\overbar w$ for the image of an element $w\in\ZZ G$ under the projection to $\ZZ H$. Then we set
\[\Polytope_\pi \coloneqq P\left(\overbar{\frac{\partial r}{\partial x}}\right) - P\big(\overbar{y-1}\big)=P\left(\overbar{\frac{\partial r}{\partial y}}\right) - P\big(\overbar{x-1}\big)\in \PolyT(H).\]
\end{defi}

It is shown in~\cite[Proposition~3.5]{FT2015} that the element $\Polytope_\pi\in\PolyT(H)$
defined in this way is indeed a single polytope.

For a nice $(2,1)$-presentation $\pi$, Friedl and Tillmann also endow $\Polytope_\pi$ with a marking of vertices, turning it into a marked polytope $\MarkedPolytope_\pi$.
A vertex of $\Polytope_\pi$ is declared marked if any of its duals contains a character lying in $\Sigma^1(G)$. Friedl--Tillmann prove in~\cite[Theorem~1.1]{FT2015} that every character lying in any dual of a marked vertex lies in $\Sigma^1(G)$, and hence the markings of $\Polytope_\pi$ and $\Sigma^1(G)$ determine one another.

If $\pi=\langle x, y\mid r\rangle$ and $\pi'=\langle x, y\mid r'\rangle$ are two $(2,1)$-presentations such that there exists an automorphism $f\colon \langle x, y\rangle\to\langle x, y\rangle$ of the free group on two generators satisfying $f(r)=r'$, then the two presentations clearly define isomorphic groups. The automorphism $f$ induces an isomorphism $\overbar f\colon H_\pi\to H_{\pi'}$ between the free parts of the abelianisations of $G_\pi$ and $G_{\pi'}$.
\begin{prop}
    \label{prop:tgor:nielsen}
    Let $\pi=\langle x, y \mid r\rangle$ and $\pi'=\langle x, y\mid r'\rangle$ be two nice $(2,1)$-presentations.
    Assume that there exists an automorphism $f\colon \langle x, y \rangle \to \langle x, y\rangle$ with $f(r)=r'$.
    Then
    \[\Polytope_{\pi'} = \Polytope_T(\overbar{f})(\Polytope_\pi)\in\PolyT(H_{\pi'}).\]
\end{prop}
\begin{proof}
    The automorphism group of a finitely generated free group is generated by the elementary Nielsen transformations, which in the case of two generators $x$ and $y$ consist of the following operations:
    \begin{enumerate}
        \item Interchange $x$ and $y$: $f_1(x)=y, f_1(y)=x$.
        \item Replace $x$ with $x^{-1}$: $f_2(x)=x^{-1}, f_2(y)=y$.
        \item Replace $x$ with $xy$: $f_3(x)=xy, f_3(y)=y$.
    \end{enumerate}
    Since the statement of the proposition is functorial in $f$, we are thus left to show that $\Polytope_{\pi'} = \Polytope_T(\overbar{f})(\Polytope_\pi)$ holds whenever $f$ is one of $f_1, f_2$ and $f_3$.

    The chain rule for Fox derivatives~\cite[(2.6)]{Fox1953} applied to $f$ takes the following form:
    \begin{equation*}
        \label{eq:tgor:chain_rule}
        \frac{\partial}{\partial x} f(r) = f\Big(\frac{\partial}{\partial x} r\Big)\cdot\frac{\partial}{\partial x} f(x) + f\Big(\frac{\partial}{\partial y} r\Big)\cdot \frac{\partial}{\partial x} f(y).
    \end{equation*}
    For the three elementary Nielsen transformations, we obtain
    \begin{alignat*}{6}
        \frac{\partial}{\partial x} f_1(r) &= f_1\Big(\frac{\partial}{\partial x} r\Big)\cdot 0 &&+ f_1\Big(\frac{\partial}{\partial y} r\Big)\cdot 1&&=f_1\Big(\frac{\partial}{\partial y} r\Big)\\
        \frac{\partial}{\partial x} f_2(r) &= f_2\Big(\frac{\partial}{\partial x} r\Big)\cdot (-x^{-1}) &&+ f_2\Big(\frac{\partial}{\partial y} r\Big)\cdot 0&&=f_2\Big(\frac{\partial}{\partial x} r\Big)\cdot (-x^{-1})\\
        \frac{\partial}{\partial x} f_3(r) &= f_3\Big(\frac{\partial}{\partial x} r\Big)\cdot 1 &&+ f_3\Big(\frac{\partial}{\partial y} r\Big)\cdot 0&&=f_3\Big(\frac{\partial}{\partial x} r\Big).
    \end{alignat*}
  When $f=f_2$ or $f=f_3$, we read off that $\partial r'/\partial x$ and $f(\partial r/\partial x)$ differ only by a factor of the form $\pm g$ for some $g\in G_{\pi'}$.
    It follows that $P(\partial r'/\partial x)$ and $\Poly(\overbar f)(P(\partial r/\partial x))$ agree up to translation and hence define the same class in $\PolyT(H_{\pi'})$.
    Since $f(y)=y$ in these cases, the same holds true for the polytopes $\Polytope_\pi$ and $\Polytope_{\pi'}$.

    For $f_1$, we obtain using~\eqref{eq:tgor:fundamental} that
    \begin{align*}
        \Polytope_{\pi'}&=P\Big(\frac{\partial}{\partial x} f_1(r)\Big) - P(y-1)
        =\PolyT(\overbar{f_1})\Big(P\Big(\frac{\partial}{\partial y} r\Big)\Big) - \PolyT(\overbar{f_1})\Big(P(x-1)\Big)\\
        &=\PolyT(\overbar{f_1})\Big(P\Big(\big(\frac{\partial}{\partial y} r\big)(y-1)\Big) - P\Big(y-1\Big) - P\Big(x-1\Big)\Big)\\
        &=\PolyT(\overbar{f_1})\Big(P\Big(\big(\frac{\partial}{\partial x} r\big)(x-1)\Big) - P\Big(x-1\Big) - P\Big(y-1\Big)\Big)\\
        &=\PolyT(\overbar{f_1})\Big(P\Big(\frac{\partial}{\partial x} r\Big) - P\Big(y-1\Big)\Big) = \PolyT(\overbar{f_1}) (\Polytope_\pi),
    \end{align*}
    which concludes the proof also in this case.
\end{proof}

There are $(2,1)$-presentations $\pi=\langle x, y\mid r\rangle$ and $\pi'=\langle x, y\mid r'\rangle$ giving rise to isomorphic groups, such that no isomorphism lifts to an automorphism of $\langle x, y\rangle$ mapping $r$ to $r'$. The first examples of such pairs of presentations appeared in \cite{MP1973}, one of which is $\langle x, y \mid x^2y^{-2}x^2y^{-3}\rangle \cong \langle x, y \mid x^2y^{-5}\rangle$. This raises the question whether the (marked) polytopes associated to $\pi$ and $\pi'$ are still related. A possible answer to this question has been formulated as a conjecture by Friedl and Tillmann:
\begin{conj}[{\cite[Conjecture~1.2]{FT2015}}]
\label{conj:tgor:invariance}
If $G$ is a group admitting a nice $(2, 1)$-presentation $\pi$, then $\MarkedPolytope_\pi\subset H_1(G;\RR)$ is an invariant of $G$ (up to translation).
\end{conj}
In more formal terms, the conjecture asserts that if $f\colon G_\pi\to G_{\pi'}$ is an isomorphism of two groups associated to $(2,1)$-presentations $\pi$ and $\pi'$, then $\Polytope_{\pi'}=\PolyT(\overbar f)(\Polytope_\pi)\in \PolyT(H_\pi)$, where $\overbar f\colon H_\pi \to H_{\pi'}$ is the isomorphism of the free parts of the abelianisations of $G_\pi$ and $G_{\pi}'$ induced by $f$.

As evidence for their conjecture, Friedl and Tillmann prove:
\begin{theo}[{\cite[Theorem~1.3]{FT2015}}]
\label{theo:tgor:residually_elementary_amenable}
If $G$ is a torsion-free group admitting a nice $(2, 1)$-presentation $\pi$ and $G$ is residually \{torsion-free elementary amenable\}, then $\MarkedPolytope_\pi\subset H_1(G;\RR)$ is an invariant of $G$ (up to translation).
\end{theo}
They further remark that the polytope does not change (up to translation) when the relator is permuted cyclically.

Making use of their construction of universal $L^2$-torsion, Friedl and Lück resolved this conjecture and provided a construction of $\MarkedPolytope_\pi$ intrinsic to the group $G$ under the additional assumption that $G$ is torsion-free and satisfies the Atiyah conjecture:
\begin{theo}[{\cite[Remark~5.5]{FL2017}}]
\label{theo:tgor:atiyah}
If $G$ is a torsion-free group admitting a nice $(2, 1)$-presentation $\pi$ and $G$ satisfies the Atiyah conjecture, then $\MarkedPolytope_\pi\subset H_1(G;\RR)$ is an invariant of $G$ (up to translation). Moreover, $\Polytope_\pi=\PLtwo(G)$.
\end{theo}

By using agrarian torsion instead of universal $L^2$-torsion, we are able to remove the additional assumptions on $G$, thereby resolving~\cref{conj:tgor:invariance}:
\begin{theo}
\label{theo:tgor:invariance}
If $G$ is a group admitting a nice $(2, 1)$-presentation $\pi$, then $\MarkedPolytope_\pi\subset H_1(G;\RR)$ is an invariant of $G$ (up to translation). Moreover, if $G$ is torsion-free then $\Polytope_\pi=\Pag{D_r}(G)\in\PolyT(\ZZ^2)$ for any choice of an agrarian embedding $\ZZ G\hookrightarrow D$.
\end{theo}
\begin{proof}
We start by looking at the case of $G$ containing torsion. The solution to this case was pointed out to the authors by Alan Logan.

First note that in this case, the BNS invariant $\Sigma^1(G)$ is empty -- this follows immediately from Brown's algorithm~\cite{Brown1987}, or equivalently, from the construction of the marking of $\MarkedPolytope_\pi$. An alternative way to see this is to observe that the first $L^2$-Betti number of $G$ is not zero, see~\cite{DicksLinnell2007}.

Since $\Sigma^1(G) = \emptyset$, we need only worry about $\Polytope_\pi$. If one alters the presentation $\pi$ by applying an automorphism $f$ of the free group $F_2 = \langle x, y\rangle$ to the relator $r$, the polytope remains invariant in the sense of~\cref{conj:tgor:invariance} by~\cref{prop:tgor:nielsen}. But it was shown by Pride~\cite{Pride1977} that when $G$ contains torsion, every two two-generator one-relator presentations of $G$ are related by an automorphism of $F_2$, up to possibly replacing the relator $r$ in one of the presentations by $r^{-1}$. This last operation does not alter the class of the polytope since, as a consequence of the product rule for Fox derivatives, we get $\partial r^{-1}/\partial x=-r^{-1} \partial r/\partial x$, and thus the polytopes associated to $\partial r^{-1}/\partial x$ and $\partial r/\partial x$ agree up to translation.

\medskip
Now suppose that $G$ is torsion-free. Then the equality $\Polytope_\pi=\Pag{D_r}(G)$ follows directly from the definitions of $\Polytope_\pi$ and $\Pag{D_r}(G)$ by the computation done in \cref{prop:tgor:polytope}, and the agrarian polytope is an invariant of the group by construction.
We conclude from~\cite[Theorem~1.1]{FT2015} that once $\Polytope_\pi$ is known to be an invariant of $G$, the same is true for the marked version $\MarkedPolytope_\pi$ since marked vertices are determined by the BNS invariant $\Sigma^1(G)$ of the group $G$.
\end{proof}

As a consequence of the equality $\Polytope_\pi=\Pag{D_r}(G_\pi)$ for a $(2,1)$-presentation $\pi$ giving rise to a torsion-free group we conclude that $\Pag{D_r}(G_\pi)$ is actually independent of the choice of agrarian embedding.

\smallskip

Friedl and Tillmann claim in~\cite[Proposition~8.1]{FT2015} and the subsequent two paragraphs that they can associate a single polytope $\Polytope_\pi$ to any $(2, 1)$-presentation $\pi=\langle x, y\mid r\rangle$ where $r$ is non-trivial and cyclically reduced, even without assuming the presentation to be nice.
If $b_1(G_\pi)=1$, $x$ represents a generator of the free part of the abelianisation of $G_\pi$ and $y$ represents the trivial element therein, they call such a presentation \emph{simple}.
For a simple presentation $\pi$, the polytope $\Polytope_\pi$ is computed by the formula involving the Fox derivative of $r$ with respect to $x$ from \cref{defi:tgor:polytope}, and therefore agrees with $\Pag{D_r}(G_\pi)$ if $G_\pi$ is torsion-free.

The statement and proof of~\cite[Proposition~8.1]{FT2015} are not fully correct, as the following example shows:
\begin{exam}
\label{exam:tgor:virtual}
Consider the simple $(2, 1)$-presentation $\pi=\langle x, y\mid y^2\rangle$.
Then the associated polytope $\Polytope_\pi$ is only a virtual polytope, more specifically the additive inverse of the class of a unit interval in $\PolyT (\ZZ) = \PolyT (\langle \overbar x\rangle)$.
\end{exam}

In the proof of~\cite[Proposition~8.1]{FT2015}, the assumption that the relator $r$ is either of the form $x^{m_1}y^{n_1}\cdots x^{m_k}y^{m_k}$ or $y^{n_1}x^{m_1}\cdots y^{m_k}x^{m_k}$ for \emph{non-zero} integers $m_1, n_1,\ldots, m_k, n_k$ is incorrect; in our example $k=1$, $m_1=0$ and $n_1=2$.

In order to fix the statement and the proof of the proposition, it is necessary to consider the case of group presentations $\langle x, y\mid y^n\rangle$, $n\in\ZZ, n\neq 0$ separately.
These presentations are the only simple ones for which any of the $m_i$ is zero.
In this case, the polytope $P(\overbar{\partial r/\partial x})$ is an interval of length $D=0$, which means that $\Polytope_\pi$ is the additive inverse of a unit interval in $\PolyT(\ZZ)$.

With this additional case considered, we now observe that the correct result of \cite[Proposition~8.1]{FT2015} should be that $\Polytope_\pi$ is a single polytope for a simple $(2, 1)$-presentation $\pi$ if and only if $G_\pi$ is not isomorpic to $\ZZ\ast\ZZ/n\ZZ$ for any $n\in\ZZ$.
The polytope $\Polytope_\pi$ can be turned into a marked polytope $\MarkedPolytope_\pi$ in the Friedl--Tillmann sense if and only if $G_\pi$ is neither isomorphic to $\ZZ\ast\ZZ/n\ZZ$ nor to $B(\pm 1, n)\coloneqq {\langle x, y\mid xy^{\pm 1}x^{-1} y^{-n}\rangle}$ for $n\in\ZZ$.

\smallskip

The problem with the Baumslag--Solitar groups $B(\pm 1, n)$ is that the resulting polytope is a singleton lying in a $1$-dimensional $\RR$-vector space. Since $\Sigma^1(B(\pm 1, n))$ is non-trivial and proper in $H^1(B(\pm 1, n);\RR)$, there is no marking of $\Polytope_\pi$ in the Friedl--Tillmann sense which would correctly control the BNS invariant. Our notion of marking of vertices of a polytope circumvents this problem, and allows for a definition of $\MarkedPolytope_\pi$ also for these groups by marking one of the duals of the only face and not marking the other.

The groups $\ZZ\ast\ZZ/n\ZZ$ arising from the presentations $\langle x, y\mid y^n\rangle$ all admit a virtual polytope which is the additive inverse of the unit interval in $\PolyT(\ZZ\langle \overbar x\rangle)$.
The notion of a marked polytope readily extends to additive inverses of single polytopes by describing a marking for the single polytope.
Since $\ZZ\ast \ZZ/n\ZZ$ is an ascending HNN extension along any of the two possible epimorphisms to $\ZZ$ if $n=\pm 1$ and contains torsion otherwise, the polytope will have all duals of its only face marked if $n=\pm 1$ and not marked if $n\neq \pm 1$.

\section{Polytope thickness and splitting complexity}

We continue with the notation of the  previous section. Our aim now is to show that the thickness of $\Polytope_\pi$ controls the minimal complexity of certain expressions of $G$ as an HNN extension over a finitely generated group. Before we state the precise connection, we need to introduce the following concept:
\begin{defi}[{\cite[Section 5.1]{FLT2016}}]
Let $\Gamma$ be a finitely presented group and let $\phi\colon\Gamma\to\ZZ$ be an epimorphism. A \emph{splitting of $(\Gamma, \phi)$} is a presentation of $\Gamma$ as an HNN extension with induced character $\phi$ and finitely generated base and associated groups.
\end{defi}

It is proved in~\cite[Theorem~A]{BS1978} that any pair $(\Gamma, \phi)$ admits a splitting. Hence we can define the \emph{splitting complexity of $(\Gamma,\phi)$} as
\[c(\Gamma,\phi)\coloneqq \min\{\rk(B) \mid (\Gamma,\phi)\text{ splits with associated group $B$}\},\]
where $\rk(B)$ denotes the minimal number of generators of $B$. We also define the \emph{free splitting complexity of $(\Gamma, \phi)$} as
\[c_f(\Gamma,\phi)\coloneqq \min\{\rk(F) \mid (\Gamma,\phi)\text{ splits with associated free group $F$}\},\]
which may be infinite. We always have $c(\Gamma,\phi)\leq c_f(\Gamma,\phi)$.

Friedl and Tillmann observed the following connection between the thickness of $\Polytope_\pi$ and the (free) splitting complexity of $G$:
\begin{theo}[{\cite[Theorem~7.2]{FT2015}}]
Let $G$ be a residually \{torsion-free elementary amenable\} group admitting a nice $(2, 1)$-presentation $\pi$. Then for any epimorphism $\phi\colon G\to\ZZ$ we have
\[c(G,\phi)=c_f(G, \phi)=\thickness_\phi(\Polytope_\pi) + 1.\]
\end{theo}
Note that every residually \{torsion-free elementary amenable\} group must itself be torsion-free. Friedl, Lück, and Tillmann then noted in~\cite[Theorem~5.2]{FLT2016} that the original proof could be adapted to the setting of~\cite{FL2016}, thereby giving the same formula for groups satisfying the Atiyah conjecture.

We will now present a common generalisation of these results. For this, we require the following strengthened form of a proposition of Harvey, which is evident from the last sentence of its original proof:
\begin{prop}[{\cite[Proposition~9.1]{Harvey2005}}]
\label{prop:tgor:harvey}
Let $D$ be a skew field and $D[t,t^{-1}]$ a twisted Laurent polynomial ring with coefficients in $D$. Let $M=A+tB$ where $A$ and $B$ are two $l \times m$ matrices over $D$. Then the map $r_M\colon D[t,t^{-1}]^l\to D[t,t^{-1}]^m$ given by right multiplication by $M$ satisfies
\[\dim_D \tors(\coker(r_M))\leq \rk_D B.\]
\end{prop}

We are now in a position to improve upon both~\cite[Theorem~7.2]{FT2015} and~\cite[Theorem~5.2]{FLT2016} by recasting the proof of~\cite[Theorem~7.2]{FT2015} in the agrarian world.
In the statement of the following theorem, the agrarian polytope $\Pag{D_r}(G)$ can be replaced by $\Polytope_\pi$ for any nice or simple $(2, 1)$-presentation $\pi$ of $G$ in the sense of \cite[Section~8.1]{FT2015}.
\begin{theo}
\label{theo:tgor:thickness}
Let $G$ be a torsion-free two-generator one-relator group other than the free group on two generators. Then for every epimorphism $\phi\colon G\to\ZZ$ we have
\[c(G,\phi)=c_f(G, \phi)=\thickness_\phi(\Pag{D_r}(G)) + 1.\]
\end{theo}
\begin{proof}
The inequality $c_f(G,\phi)\leq\thickness_\phi(\Polytope_\pi)+1$ is proved in~\cite[Proposition~7.3]{FT2015} for all nice $(2,1)$-presentations.
The proof of~\cite[Lemma~7.5]{FT2015} also applies to any simple $(2,1)$-presentation $\langle x, y\mid r\rangle$ for which $r$ is not a word in just one of the generators and its inverse, since then the numbers $m_1$ and $n_1$ appearing in the proof are non-zero.
Any other simple $(2,1)$-presentation $\pi$ is, up to renaming the generators, of the form $\langle x, y\mid x^n\rangle$ for $n\in\NN, n\neq 0$, and there are only two different epimorphisms $G_\pi\to\ZZ$.
It is then easy to see right from the definitions that the splitting complexity and thickness of $\Polytope_\pi$ with respect to any of the two epimorphisms are given by $0$ and $-1$, respectively.

Since every torsion-free two-generator one-relator group $G$ that is not the free group on two generators admits either a nice or a simple presentation $\pi$ and $\Pag{D_r}(G)=\Polytope_\pi$ by \cref{prop:tgor:polytope}, we are left to show that $c(G,\phi)\geq\thickness_\phi(\Pag{D_r}(G))+1$. By \cref{theo:twisted:thickness}, this is further reduced to the following statement about the $\phi$-twisted $D_r$-agrarian Euler characteristic of $G$:
\[c(G,\phi)-1 \geq - \chiag[D_r](G;\phi).\]

Recall from the proof of \cref{lemm:tgor:torsion} that the Cayley 2-complex $X$ associated to a $(2, 1)$-presentation of $G$ serves as a model of $EG$ and that the application of \cref{theo:twisted:thickness} is justified since we constructed a non-degenerate $\tau$-chain. By \cref{lemm:twisted:kernel,lemm:twisted:laurent_computation}, we can thus compute $\chiag[D_r](G;\phi)$ from the Betti numbers of the complex $D_r[t, t^{-1}]_\phi\otimes C_*(X)$:
\[D_r[t, t^{-1}]_\phi \xrightarrow{\begin{pmatrix}\frac{\partial r}{\partial x} & \frac{\partial r}{\partial y}\end{pmatrix}}  D_r[t, t^{-1}]_\phi ^2\xrightarrow{\begin{pmatrix}x-1 \\ y-1 \end{pmatrix}} D_r[t, t^{-1}]_\phi.\]
Since $D_r[t, t^{-1}]_\phi$ is a (non-commutative) principal ideal domain, the kernel of the differential originating from degree 2 is free. It is also seen to be torsion by \cref{lemm:twisted:laurent_computation} and hence $\dim_{D_r} H_p(D_r[t, t^{-1}]_\phi\otimes C_*(X))=0$ for $p\geq 2$.

We let $c=c(G,\phi)$ and choose a splitting
\[G=\langle A, t\mid \mu(B)=tBt^{-1}\rangle\]
of $(G,\phi)$ with associated group $B$ generated by $x_1,\dots,x_c$; in particular $A\subseteq\ker(\phi)$ is finitely generated. We pick a presentation $A=\langle g_1,\dots, g_k\mid r_1, r_2,\dots\rangle$, which is possible since $G$ and thus $A$ are countable. Denote the number of relations in this presentation by $l\in\ZZ_{\geq 0}\cup\{\infty\}$. The splitting of $(G,\phi)$ then gives the following alternative presentation of $G$:
\[G=\langle g_1, \dots, g_k, t\mid r_1,r_2,\dots, \mu(x_1)^{-1}tx_1t^{-1},\dots, \mu(x_c)^{-1}tx_ct^{-1}\rangle.\]
Note that the words $r_i$, $x_j$ and $\mu(x_j)$ are words in the generators $g_i$ of $A$.
Denote by $Y$ the Cayley 2-complex associated to this presentation. By construction, $\pi_1(Y/G)=\pi_1(X/G)$, and thus $Y$ can be turned into a model for $EG$ by attaching $G$-cells in dimension $3$ and higher only. Hence, its homology with arbitrary coefficients agrees with that of $X$ up to dimension $1$, which in particular implies that $\dim_{D_r} H_p(D_r[t, t^{-1}]_\phi\otimes C_*(X))=\dim_{D_r} H_p(D_r[t, t^{-1}]_\phi\otimes C_*(Y))$ for $p=0,1$.

In conclusion, we will know $\chiag[D_r](G;\phi)$ if we compute the first two $D_r[t, t^{-1}]_\phi$-Betti numbers of the $G$-CW-complex $Y$. For this, we need to consider its shape in more detail.
The complex $Y$ is a two-dimensional free $G$-CW-complex with one zero-cell, $k+1$ one-cells and $l+c$ two-cells, and its cellular chain complex takes the form
\[\dots\to0\to\ZZ G^{l+c}\xrightarrow{\begin{pmatrix}M_0 & M_1\end{pmatrix}} \ZZ G\oplus \ZZ G^k \xrightarrow{\begin{pmatrix}v_0\\v_1\end{pmatrix}}\ZZ G,\]
where the (potentially infinite) block matrix $M = \begin{pmatrix}M_0&M_1\end{pmatrix}$ representing the second differential consists of the Fox derivatives of the relations with respect to $t$ and the $g_i$, respectively, and $v_0=t-1,v_1=(g_1-1,\dots,g_k - 1)^t$. Since the relations $r_1, r_2, \dots$ are words in $\ZZ A$, their Fox derivatives with respect to $t$ are trivial and their derivatives with respect to each $g_i$ again lie in $\ZZ A$. For the other relations, we obtain
\begin{align*}
\frac{\partial}{\partial t} (\mu(x_j)^{-1} t x_j t^{-1}) & =\mu(x_j)^{-1}-\mu(x_j)^{-1}tx_jt^{-1}\in\ZZ A \textrm{ and}\\
\frac{\partial}{\partial g_i} (\mu(x_j)^{-1} t x_j t^{-1}) & = \frac{\partial}{\partial g_i} (\mu(x_j)^{-1})+\mu(x_j)^{-1}t\frac{\partial}{\partial g_i} x_j \in \ZZ A+t\cdot \ZZ A.
\end{align*}

Hence, the matrix $M$ is of the shape

\[
\begin{tikzpicture}[baseline=0cm,decoration=brace]
    \matrix (m) [matrix of math nodes, left delimiter=(,right delimiter=)] {
        0&\vphantom{0}&&\vphantom{0}\\
        \smashedvdots&&\in \ZZ A&\\
        0&&&\\
        \in \ZZ A&&&\\
        \smashedvdots&&\in \ZZ A+ t\cdot \ZZ A&\\
        \in \ZZ A&&&\\
    };
    \draw[decorate, transform canvas={xshift=-2.8em}, thick] (m-3-1.south west)+(0,1mm) -- node[left=2pt] {$l$} (m-1-1.north west);
    \draw[decorate, transform canvas={xshift=-1.9em}, thick] (m-6-1.south west) -- node[left=2pt] {$c$} (m-4-1.north west)+(0,1mm);
    \draw[decorate, transform canvas={yshift=2em}, thick] (m-1-2.south west) -- node[above=2pt] {$k$} (m-1-4.south east);
\end{tikzpicture}
\]
with the block $M_0$ consisting of the first column of $M$. Now consider the following chain map of $D_r[t, t^{-1}]_\phi$-chain complexes, where the vertical maps are given by projections and both complexes continue trivially to the left and right:

\[
\mathclap{
\begin{tikzcd}[ampersand replacement=\&, row sep=large, column sep=large]
D_r[t, t^{-1}]_\phi^{l+c} \arrow[d]\arrow{r}{\begin{pmatrix}M_0 & M_1\end{pmatrix}} \& D_r[t, t^{-1}]_\phi\oplus D_r[t, t^{-1}]_\phi^k \arrow[d]\arrow{r}{\begin{pmatrix}v_0\\v_1\end{pmatrix}} \& D_r[t, t^{-1}]_\phi \arrow[d]\\
0 \arrow[r] \& D_r[t, t^{-1}]_\phi^k / (D_r[t, t^{-1}]_\phi^{l+c} M_1)\arrow[r,"\begin{pmatrix}v_1\end{pmatrix}"]\& D_r[t, t^{-1}]_\phi / (t-1)
\end{tikzcd}
}
\]

Since multiplication from the right with $t-1$ is injective on $D_r[t, t^{-1}]_\phi$, the chain map induces an isomorphism on homology in degrees 0 and 1. Since all the homology modules $H_i(D_r[t,t^{-1}]_\phi\otimes C_*(Y))$ are torsion by \cref{lemm:twisted:laurent_computation}, the same holds true for the homology of the lower chain complex. Using \cref{prop:tgor:harvey}, we thus get the bound
\[
\dim_{D_r} D_r[t, t^{-1}]_\phi^k / (D_r[t, t^{-1}]_\phi^{l+c} M_1) = \dim_{D_r} \tors(\coker(r_{M_1}))\leq c.
\]
As $\deg(t-1)=1$, we also get
\[\dim_{D_r} D_r[t, t^{-1}]_\phi / (t-1) = 1.\]

In particular, the lower chain complex consists of finite $D_r$-vector spaces. Applying the rank-nullity theorem to its only non-trivial differential, we obtain
\begin{align*}
\thickness_\phi(\Polytope_\pi) &= -\chiag[D_r](X;\phi) \\
&= \dim_{D_r} H_1(D_r[t, t^{-1}]_\phi\otimes C_*(Y)) - \dim_{D_r} H_0(D_r[t, t^{-1}]_\phi\otimes C_*(Y)) \\
&= \dim_{D_r} D_r[t, t^{-1}]_\phi^k / (D_r[t, t^{-1}]_\phi^{l+c} M_1) - \dim_{D_r} D_r[t, t^{-1}]_\phi / (t-1)\\
&\leq c-1.\qedhere
\end{align*}
\end{proof}

\begin{exam}
For words $x, y\in \langle a, b\rangle$, we define $x^y\coloneqq y^{-1}xy$ and $[x,y]\coloneqq x^{-1}y^{-1}xy$. Consider the two-generator one-relator group $G$ defined by
\[\left\langle a, b\;\middle|\; [a, b]=\left[[a, b], [a, b]^b\right]\right\rangle,\]
which can be presented in cyclically reduced form as
\[\pi\coloneqq \left\langle a, b\;\middle|\; a^{-1}bab^{-1}a^{-1}bab^{-2}a^{-1}baba^{-1}b^{-2}ab\right\rangle.\]
We see directly from the first presentation of $G$ that the relator becomes trivial in the abelianisation, hence $b_1(G)=2$ and $\pi$ is a nice $(2, 1)$-presentation. By~\cite[Proposition~II.5.18]{LS2001}, the group $G$ is also torsion-free since the single relator is not a proper power.

We claim that $G$ is not residually solvable, i.e., not every element maps non-trivially into a solvable quotient of $G$. Since the element $[a,b]$ can be written as an arbitrarily deeply nested iterated commutator (using the relation of the first presentation above), it is contained in all derived subgroups of $G$ and hence of every quotient. But if a quotient is solvable, some derived subgroup and hence the image of $[a, b]$ will be trivial. It is thus left to show that $[a,b]$ is non-trivial in $G$. Assume that $[a,b]=1$ in $G$. Then $G$ is abelian and hence also $[b,a]=b^{-1}a^{-1}ba=1$ in $G$. But $[b,a]$ appears as a proper subword of the relator in $\pi$ and thus represents a non-trivial element by~\cite[Proposition~II.5.29]{LS2001}.

We conclude that a method such as the one employed in~\cite[Lemma~6.1]{FT2015} cannot be used to deduce that $G$ is residually \{torsion-free elementary amenable\} and hence satisfies the assumptions of \cref{theo:tgor:residually_elementary_amenable}. We deem it plausible that $G$ is even not residually \{torsion-free elementary amenable\} and is thus not covered by \cref{theo:tgor:residually_elementary_amenable}, but to the best of the authors' knowledge no two-generator one-relator group has been shown to have this property.

If we denote the single relator of $\pi$ by $r$, an easy but tedious computation shows that
\begin{align*}
\frac{\partial r}{\partial a} =& -\overbrace{b^{-1}a^{-1}}^{(-1,-1)} + \overbrace{b^{-1}a^{-1}b}^{(-1, 0)} - \overbrace{b^{-1}a^{-1}bab^{-1}a^{-1}}^{(-1,-1)} + \overbrace{b^{-1}a^{-1}bab^{-1}a^{-1}b}^{(-1,0)} \\
&-\overbrace{(b^{-1}a^{-1}ba)^2b^{-2}a^{-1}}^{(-1,-2)} + \overbrace{(b^{-1}a^{-1}ba)^2b^{-2}a^{-1}b}^{(-1, -1)}-\overbrace{(b^{-1}a^{-1}ba)^2b^{-2}a^{-1}baba^{-1}}^{(-1,0)}\\
&+\overbrace{(b^{-1}a^{-1}ba)^2b^{-2}a^{-1}baba^{-1}b^{-2}}^{(-1,-2)},
\end{align*}
with the image in the abelianisation of each summand noted in brackets. The convex hull of these points in $\RR^2$ corresponds to an interval of length 2 in the $b$-direction, hence $\Polytope_\pi=\Pag{D_r}(G)$ is an interval of length 1 in the $b$-direction. The marked polytope $\MarkedPolytope_\pi$ has no markings since all abelianised monomials appear multiple times.

Let $\phi_b\colon G\to \ZZ$ be the homomorphism sending $a$ to $0$ and $b$ to $1$. Since $\thickness_{\phi_b}(\Pag{D_r}(G))=1$, we conclude from \cref{theo:tgor:thickness} that $c_f(G,\phi_b)=c(G,\phi_b)=2$. A (free) splitting of $G$ along $\phi_b$ of minimal rank is thus given by
\begin{align*}
G=&\left\langle a, b, x, y\;\middle|\; x=[x,y], y=x^b, x=[a, b]\right\rangle\\
=& \left\langle a, x, y, b\;\middle|\; x=[x,y], y=x^b, ax=a^b \right\rangle.
\end{align*}

Note that our example is a nice version of the original example of a two-generator one-relator group which is not residually finite produced by Baumslag in~\cite{Baumslag1969}.
\end{exam}

\bibliography{main}
\end{document}